\documentclass[11pt]{amsart}

\usepackage[utf8x]{inputenc} 
\usepackage[a4paper]{geometry} 
\usepackage{enumerate}  
\usepackage{amsmath} 
\usepackage{amssymb}
\usepackage{color,colortbl}
\usepackage{multirow}
\usepackage{hyperref}  
\usepackage{verbatim}
                 \hypersetup{ pdfborder={0 0 0}, 
                              colorlinks=true, 
                              linktoc=page,
                              pdfauthor={Giovanni Staglian{\`o}}, 
                              pdftitle={Special cubic birational transformations of projective spaces}} 
\definecolor{Gray}{gray}{0.9}                            
\usepackage{mathptmx}                      
   
\theoremstyle{plain} 
\newtheorem{proposition}{Proposition}[section] 
\newtheorem{theorem}[proposition]{Theorem} 
\newtheorem{lemma}[proposition]{Lemma} 
\newtheorem{corollary}[proposition]{Corollary} 
 
\theoremstyle{definition}

\newtheorem{example}[proposition]{Example} 
\theoremstyle{remark} 
\newtheorem{remark}[proposition]{Remark} 
                              
\newcommand{\I}{{\mathcal{I}}}     
\renewcommand{\O}{{\mathcal{O}}}   
\newcommand{\T}{{\mathcal{T}}}    
\newcommand{\E}{{\mathcal{E}}}    
\newcommand{\B}{{\mathfrak{B}}} 
\newcommand{\R}{{\mathfrak{R}}} 
\newcommand{\N}{{\mathcal{N}}}     
 
\newcommand{\ZZ}{{\mathbf{Z}}} 
      
\newcommand{\PP}{{\mathbb{P}}}

\numberwithin{equation}{section}

\title{Special cubic birational transformations of projective spaces} 
\author[G. Staglian\`o]{Giovanni Staglian\`o} 
\address{Dipartimento di Matematica e Informatica, Universit\`{a} degli Studi di Catania, Viale A. Doria 5, 95125 Catania, Italy}
\date{\today} 
\email{\href{mailto:giovannistagliano@gmail.com}{giovannistagliano@gmail.com}} 

\begin{document}

\begin{abstract} 
We extend our classification 
of special Cremona transformations 
whose base locus has dimension at most three 
to the case when 
the target space 
 is replaced by a (locally) factorial complete intersection.
\end{abstract}

\maketitle

\section*{Introduction}
A birational transformation of a complex projective space into a not too singular irreducible projective variety
is called \emph{special} if its base locus scheme is smooth and connected.
The study of special birational transformations started by Semple and  Tyrrell
in \cite{cite1-semple,semple-tyrrell,cite2-semple}, who in particular constructed  
the main examples of special Cremona transformations having two-dimensional base locus.
Later, a more systematic treatment has been provided by Crauder, Katz, Ein, and Shepherd-Barron in
\cite{katz-cubo-cubic,crauder-katz-1989,ein-shepherdbarron,crauder-katz-1991},
and further contributions has been given in \cite{hulek-katz-schreyer}.
More recently, 
other authors and ourselves 
obtained more results on the theory, 
focusing mainly on classification problems; see 
\cite{russo-qel1,note,alzati-sierra-quadratic,note2,note3,li2016,alzati-sierra,fu-hwang,note4}.

In this paper, we consider the problem of classifying special birational 
transformations whose base locus has dimension at most three.
When the base locus has dimension three, a preliminary analysis shows that
there are only three important classes of transformations (see Corollary~\ref{casiProp}):
\begin{itemize}
 \item quintic  transformations of $\PP^5$;
 \item quadratic transformations of $\PP^8$;
 \item cubic  transformations of $\PP^6$.
\end{itemize}
This situation was basically already obtained in \cite{crauder-katz-1991},
and the classification of the special Cremona transformations as in the first case
has been achieved in \cite[Theorem~3.2]{ein-shepherdbarron}.
Recently, Alzati and Sierra in \cite{alzati-sierra} 
extended 
this  result  to the case in which 
the image $\ZZ$ of the map is a prime Fano manifold.

The transformations as in the second case have been 
 classified by ourselves 
 under the hypothesis that the image $\ZZ$ is a (locally) factorial variety;
 see \cite{note,note2,note3}. In the particular case when $\ZZ$
 is a factorial complete intersection, 
 the classification is summarized in Table~\ref{table: SCTintComplQuad}.

 In \cite{note4} we initiated the study of the transformations as in the third and last case,
 by classifying the special cubic Cremona transformations of $\PP^6$,
 and hence completing the classification
 of the special Cremona transformations with three-dimensional base locus.
 This paper 
 is a continuation to \cite{note4}. Here we classify 
 special cubic birational transformations of $\PP^6$
 into a factorial complete intersection $\ZZ$, and hence we complete
 the classification of all special birational transformations 
 into a factorial complete intersection
 and whose base locus 
 has dimension at most three.
  The main result of this paper is the following:
  \begin{theorem}
   There are exactly $28 = 9 + 15 + 4$ types (listed, respectively, in Tables~\ref{table: SCTintCompl}, \ref{table: SCTintComplQuad}, and \ref{table: SCTintComplQuint})
   of special birational transformations of a projective space
   into a factorial complete intersection 
   such that the base locus has dimension at most three and the inverse map is not linear.
  \end{theorem}

 The choice of our hypothesis on $\ZZ$ 
  is motivated not only by the desire of simplifying the problem, 
  but also because 
  this condition is satisfied in most of the known examples; see also Remarks~\ref{rem Fano}, \ref{rem quadratic},
  and \ref{rem quintic}.
 
 We give also results on special birational transformations whose inverse map 
 is defined by linear forms, even if these transformations  from our point of view 
 are less interesting, and 
 never occur in the case of 
  Cremona transformations. We have the following:
    \begin{theorem}
   There are exactly $9$ types (listed in Table~\ref{table: SCTlinear})
   of special birational transformations of a projective space
   into a factorial complete intersection 
   such that the base locus has dimension at most three and the inverse map is linear.
  \end{theorem}
The paper is organized as follows. In Section~\ref{sec prima},
by following 
arguments from \cite{crauder-katz-1989,ein-shepherdbarron},
we provide some
general 
results on special birational transformations, 
giving an account of what is known and what is unknown 
about the problem of classification when the base locus has dimension 
at most three.
In Section~\ref{cubiche} 
we focus on the main contribution of this  paper,
by classifying special cubic birational transformations into a factorial complete intersection
and whose base locus has dimension at most three, 
see 
Corollary~\ref{cor main}.
In Section~\ref{Examples}, in order to make 
our classification effective,
we construct new examples of special cubic birational transformations,
see Examples~\ref{cubo-cubic into cubic}, \ref{cubo-quartic into quadric},  \ref{example G-1-5};
we also review some constructions in \cite{note4} such as that of
a special cubic Cremona transformation having a conic bundle as base locus, see
Example~\ref{example Cremona 13-12}. 
We provide explicit equations defined over $\mathbb{Q}$ for all our examples.
In Section~\ref{cubicFourfolds} we illustrate briefly a connection between 
the study of 
special cubic birational transformations of $\PP^6$
and the celebrated problem about the rationality of cubic fourfolds.
In Section~\ref{sec quadratic} we recall the analogous classification results for
 special quadratic transformations, and in Section~\ref{quintiche} 
we analyze the special transformations of $\PP^5$ into a complete intersection, 
which was classified in \cite{alzati-sierra} under the hypothesis 
that the image $\ZZ$ of the map is a prime Fano manifold.
Since we allow singularities on $\ZZ$, we get a new type of transformation 
with respect to their classification, which is the restriction 
to a general hyperplane of the inverse of a special cubic \emph{Pfaffian} Cremona transformation of $\PP^6$.
Finally, in Section~\ref{sec linear}
we give some results on transformations having linear inverse.

\section{General results on special birational transformations}\label{sec prima}

Let $\varphi:\PP^n\dashrightarrow\ZZ\subseteq\PP^{N}$ be a birational transformation into a
factorial projective variety
$\ZZ\subseteq\PP^N$.
Without loss of generality we can always take $\ZZ\subseteq\PP^N$ to be non-degenerate
(\emph{i.e.}, not contained in any hyperplane).
We say that $\varphi$ is \emph{special} if its base locus $\B\subset\PP^n$
is smooth and connected. 
In this case, the singular locus of $\ZZ$ 
is set-theoretically contained in the base locus of $\varphi^{-1}$, and we 
 also always require that this inclusion is strict, \emph{i.e.}
that $\ZZ$ is ``not too singular''. 

\begin{remark}\label{rem 1}
In this paper we are mainly interested in the case when $\ZZ\subseteq\PP^{n+a}$ is an
$n$-dimensional non-degenerate complete intersection,
hence it is 
an intersection of hypersurfaces of degrees $e_1,\ldots,e_a\geq 2$.
We recall here some well-known facts.

If the singular locus of $\ZZ$ has dimension less than  $n-3$, then $\ZZ$ is factorial
(cf. Grothendieck's parafactoriality theorem (\emph{Samuel's conjecture}) \cite[XI, 3.14]{sga2}).
In this case, it is normal and hence projectively normal (cf. \cite[Exercise~II 8.4]{hartshorne-ag}).
Moreover, we can assume $n\geq 3$, so that we have that $\mathrm{Cl}(\ZZ)\simeq\mathrm{Pic}(\ZZ)\simeq\mathbb{Z}\langle \mathcal{O}_{\ZZ}(1)\rangle$
(cf. \cite[Corollary~3.2]{hartshorne-ample}), and hence we also have 
that the homogeneous coordinate ring of $\ZZ$, $S(\ZZ)$, is a unique factorization domain (cf. \cite[Exercise~II 6.3]{hartshorne-ag}).

This implies that if $\varphi:\PP^n\dashrightarrow\ZZ\subseteq\PP^{n+a}$
is a (special) birational map into a factorial complete intersection $\ZZ$, then 
the inverse map of $\varphi$ is uniquely represented up to 
proportionality  by a vector $(F_0,\ldots,F_n)\in S(\ZZ)^{n+1}$ 
of forms of the same degree $d=d_2$ without common factors; in other words the degree sequence of $\varphi^{-1}$ 
has length one and consists of $d$ only, see \cite{Simis2004162}.

Finally, let us recall that if $U$ is the smooth locus of $\ZZ$, we have
$\omega_{U} =
\mathcal{O}_{U}(c - n - 1)$, where $c=\Sigma_{i=1}^{a} e_i  - a $ is the 
\emph{coindex} of $\ZZ$. 
We have $c\geq 0$; moreover,  $c=0$ if and only if $\ZZ=\PP^n$, and $c=1$ if and only if $\ZZ$ is a quadric hypersurface 
(see also \cite{kobayashi-ochiai}). 
\end{remark}

We say that a special
birational transformation $\varphi:\PP^n\dashrightarrow\ZZ$ 
is \emph{of type} $(d_1,d_2)$ 
if its multidegree has the form $(1,d_1,\ldots,d_2\,\Delta,\Delta)$, where $\Delta=\deg(\ZZ)$ and $d_2\in \mathbb{Z}$.
By Remark~\ref{rem 1}, when $\ZZ$ is a factorial complete intersection then
the type of $\varphi$ is always well-defined, \emph{i.e.} $d_2\in\mathbb{Z}$.
We shall also say that $\varphi$ is quadro-quadric, quadro-cubic, cubo-quadric, and so on to indicate that 
its type is $(2,2)$, $(2,3)$, $(3,2)$, and so on.

The following proposition 
is a simple generalization of results from \cite{ein-shepherdbarron,crauder-katz-1989}.
\begin{proposition}\label{PropDim}
 Let $\varphi:\PP^n\dashrightarrow\ZZ$ be a special birational transformation of type $(d_1,d_2)$ 
 into a factorial variety $\ZZ$. 
 Denoting by $r$ (respectively by $r'$) the dimension of the base locus of $\varphi$ (respectively of $\varphi^{-1}$), we have:
 \begin{equation}\label{formulasDim}
 r  = \frac{n d_1 d_2 - n d_2 - d_1 d_2 - d_2 - c + 2}{d_1 d_2-1},\quad  
 r' = (n - r - 1) d_1 - 2,
\end{equation}
or equivalently 
\begin{equation}\label{formulasDim2}
 d_1 = \frac{r' + 2}{n- r -1},\quad d_2 = \frac{r -c + 2}{n -r' -1} ,
\end{equation}
where $c$ is a non-negative integer defined in the proof below.
\end{proposition}
\begin{proof}
Let $\B\subset\PP^n$ be the base locus of $\varphi$.
Denote by 
$\pi:\widetilde{\PP^n}=\mathrm{Bl}_{\B}(\PP^n)\to\PP^n$ the 
natural projection of the blow-up of $\PP^n$ along $\B$, $E=\pi^{-1}(\B)$ the exceptional divisor,
and $H$ a divisor in $|H^0(\pi^{\ast}(\O_{\PP^n}(1)))|$. 
Let also $\pi':\widetilde{\PP^n}\to\ZZ$ be the composition of $\pi$ and $\varphi$, 
 $\B'\subset\ZZ$ the base locus scheme of $\varphi^{-1}$, $E'=\pi'^{-1}(\B')$, and $H'$
 a divisor in $|H^0(\pi'^{\ast}(\O_{\ZZ}(1)))|$. 
 Since $\mathrm{Sing}(\ZZ)_{\mathrm{red}}\subsetneq\B'_{\mathrm{red}}$, 
 from the same argument used in the proof of \cite[Proposition~2.1]{ein-shepherdbarron} (see also \cite[Proposition~4.1]{note}),
 we deduce that
 \begin{equation}\label{blowupPIC}
 \mathrm{Pic}(\widetilde{\PP^n})=\mathbb{Z}\langle H\rangle \oplus \mathbb{Z}\langle E \rangle 
 = \mathbb{Z}\langle H'\rangle \oplus \mathbb{Z}\langle E' \rangle,
 \end{equation}
and we have the change of basis formulas:
 \begin{equation}\label{changeBasis}
 \bgroup\begin{pmatrix}
     H'\\
     E'\\
     \end{pmatrix}\egroup
 =
 \bgroup\begin{pmatrix}d_1&
     -1\\
     d_1 d_2 -1&
     -d_2\\
     \end{pmatrix}\egroup
     \bgroup\begin{pmatrix}
     H\\
     E\\
     \end{pmatrix}\egroup .
 \end{equation}
 Let us notice that 
 the restriction of $\pi'$ induces an isomorphism 
 of $\widetilde{\PP^n}\setminus E'$ onto $\ZZ\setminus \B'$
 and that $\mathrm{codim}_{\ZZ}(\B')\geq 2$.
 Thus we have
$\mathrm{Pic}(\ZZ) \simeq \mathrm{Pic}(\ZZ\setminus \B') \simeq \mathrm{Pic}(\widetilde{\PP^n}\setminus E')
 \simeq \mathbb{Z}\langle \mathcal{O}_{\widetilde{\mathbb{P}^n}}(H') \rangle $, and we can define 
 the integer $c=c(\ZZ)$ such that
 $\omega_{\ZZ\setminus\B'} \simeq \mathcal{O}_{\ZZ\setminus\B'}(c-n-1)$.
Using this invariant we can write
 the canonical divisor of $\widetilde{\PP^n}$ 
in two different ways (see again the argument in the proof of \cite[Proposition~2.1]{ein-shepherdbarron}):
\begin{equation}\label{canonicalBlowUp}
 K_{\widetilde{\PP^n}} = (-n-1) H + (n- r -1) E = (c-n-1) H' + (n- r' -1) E' .
\end{equation}
The formulas \eqref{formulasDim} and \eqref{formulasDim2} follow directly from
 \eqref{blowupPIC}, \eqref{changeBasis}, and \eqref{canonicalBlowUp}.
\end{proof}

The following result extends that of
\cite[Corollary~1]{crauder-katz-1991}, 
and in the case when $\ZZ$ is a prime Fano manifold  it is a consequence 
of  \cite[Propositions~5, 8, 11]{alzati-sierra}.
Its proof follows immediately  from Proposition~\ref{PropDim}
by taking into account that the dimensions $r$ and $r'$ are at most $n-2$.
See also Table~\ref{table: preliminaryClass}.

\begin{table}[htbp]
\centering
\begin{tabular}{lllllll}
 \hline 
 Case & $n$ & $r$  & $r'$ & $d_1$ & $d_2$ & $c$ \\  
 \hline 
 \eqref{nonquad1} & $3$ & $1$ & $1$ & $3$ & $2$ & $1$ \\
 \eqref{cre1} & $3$ & $1$ & $1$ & $3$ & $3$ & $0$ \\
 \eqref{quad1} & $4$ & $1$ & $2$ & $2$ & $2$ & $1$ \\
 \eqref{cre2} & $4$ & $1$ & $2$ & $2$ & $3$ & $0$ \\
 \hline 
 \eqref{cre3} & $4$ & $2$ & $1$ & $3$ & $2$ & $0$ \\
 \eqref{nonquad2} & $4$ & $2$ & $2$ & $4$ & $2$ & $2$ \\
\eqref{nonquad2} & $4$ & $2$ & $2$ & $4$ & $3$ & $1$ \\
 \eqref{cre4}  & $4$ & $2$ & $2$ & $4$ & $4$ & $0$ \\
\eqref{cre5} & $5$ & $2$ & $2$ & $2$ & $2$ & $0$ \\
 \eqref{quad2} & $6$ & $2$ & $4$ & $2$ & $2$ & $2$ \\
  \eqref{quad2} & $6$ & $2$ & $4$ & $2$ & $3$ & $1$ \\
  \eqref{cre6} & $6$ & $2$ & $4$ & $2$ & $4$ & $0$ \\
  \hline 
\eqref{caso3propA} & $5$ & $3$ & $2$ & $4$ & $2$ & $1$ \\
 \eqref{caso3propANcre} & $5$ & $3$ & $3$ & $5$ & $2$ & $3$ \\
 \eqref{caso3propANcre} & $5$ & $3$ & $3$ & $5$ & $3$ & $2$ \\
 \eqref{caso3propANcre} & $5$ & $3$ & $3$ & $5$ & $4$ & $1$ \\
  \eqref{caso3propAcre} & $5$ & $3$ & $3$ & $5$ & $5$ & $0$ \\
   \eqref{caso3propBNcre} & $6$ & $3$ & $4$ & $3$ & $2$ & $3$ \\
   \eqref{caso3propBNcre} & $6$ & $3$ & $4$ & $3$ & $3$ & $2$ \\
   \eqref{caso3propBNcre} & $6$ & $3$ & $4$ & $3$ & $4$ & $1$ \\
    \eqref{caso3propBcre} & $6$ & $3$ & $4$ & $3$ & $5$ & $0$ \\
     \eqref{caso3propC} & $7$ & $3$ & $4$ & $2$ & $2$ & $1$ \\
      \eqref{caso3propDNcre} & $8$ & $3$ & $6$ & $2$ & $2$ & $3$ \\
      \eqref{caso3propDNcre} & $8$ & $3$ & $6$ & $2$ & $3$ & $2$ \\
      \eqref{caso3propDNcre} & $8$ & $3$ & $6$ & $2$ & $4$ & $1$ \\
      \eqref{caso3propDcre}  & $8$ & $3$ & $6$ & $2$ & $5$ & $0$ \\
      \hline 
\end{tabular}
 \caption{Preliminary classification of special birational transformations 
 of type $(d_1,d_2)$ into a factorial variety, with $d_2>1$ and  $r\leq 3$
 (the notation is as in Proposition~\ref{PropDim}).} 
\label{table: preliminaryClass} 
\end{table}

\begin{corollary}\label{casiProp}
 Let $\varphi:\PP^n\dashrightarrow\ZZ$ be a special birational transformation of type $(d_1,d_2)$, with $d_2 > 1$,
 into a factorial variety $\ZZ$.
\begin{enumerate}
 \item If the base locus of $\varphi$ has dimension one, then one of the following holds:
 \begin{enumerate}
  \item\label{nonquad1} $n=3$ and $\varphi$ is a cubo-quadric transformation into a quadric hypersurface;
  \item\label{cre1} $n=3$ and $\varphi$ is a cubo-cubic Cremona transformation;
  \item\label{quad1} $n=4$ and $\varphi$ is a quadro-quadric transformation into a quadric hypersurface;
  \item\label{cre2} $n=4$ and $\varphi$ is a quadro-cubic Cremona transformation. 
 \end{enumerate}
\item If the base locus of $\varphi$ has dimension two, then one of the following holds:
 \begin{enumerate}
  \item\label{cre3} $n=4$ and $\varphi$ is a cubo-quadric Cremona transformation;
  \item\label{nonquad2} $n=4$ and $\varphi$ is a quartic transformation with $(d_2,c)\in\{(2,2),(3,1)\}$;
  \item\label{cre4} $n=4$ and $\varphi$ is a quarto-quartic Cremona transformation;
  \item\label{cre5} $n=5$ and $\varphi$ is a quadro-quadric Cremona transformation;
  \item\label{quad2} $n=6$ and $\varphi$ is a quadratic transformation with  $(d_2,c)\in\{(2,2),(3,1)\}$;
  \item\label{cre6} $n=6$ and $\varphi$ is a quadro-quartic Cremona transformation.
\end{enumerate}
 \item If the base locus of $\varphi$ has dimension three, then one of the following holds:
  \begin{enumerate}
  \item\label{caso3propA} $n=5$ and $\varphi$ is a quarto-quadric transformation into a quadric hypersurface;
  \item\label{caso3propANcre} $n=5$ and $\varphi$ is a quintic transformation with  $(d_2,c)\in\{(2,3),(3,2),(4,1)\}$;
  \item\label{caso3propAcre} $n=5$ and $\varphi$ is a quinto-quintic Cremona transformation;
    \item\label{caso3propBNcre} $n=6$ and $\varphi$ is a cubic transformation with  $(d_2,c)\in \{(2,3),(3,2),(4,1)\}$;
 \item\label{caso3propBcre} $n=6$ and $\varphi$ is a cubo-quintic Cremona transformation;
 \item\label{caso3propC} $n=7$ and $\varphi$ is a quadro-quadric transformation into a quadric hypersurface;
   \item\label{caso3propDNcre} $n=8$ and $\varphi$ is a quadratic transformation with $(d_2,c)\in\{(2,3),(3,2),(4,1)\}$;
\item\label{caso3propDcre} $n=8$ and $\varphi$ is a quadro-quintic Cremona transformation.
   \end{enumerate}
   \end{enumerate}
\end{corollary}

\begin{remark}\label{analisi}
The special Cremona transformations having base locus of dimension at most two 
have been classified in \cite{crauder-katz-1989} (see also \cite{katz-cubo-cubic,crauder-katz-1991,hulek-katz-schreyer}).
So we have
 a classification of the transformations as in the cases 
 \eqref{cre1}, \eqref{cre2}, \eqref{cre3}, \eqref{cre4},  \eqref{cre5}, and  \eqref{cre6}
 of Corollary~\ref{casiProp}.
 Moreover, 
 the classification of the transformations as in case \eqref{caso3propAcre} 
 follows from  \cite[Theorem~3.2]{ein-shepherdbarron}.
 
The special quadratic birational transformations having 
 base locus of dimension at most three
have been classified in \cite{note,note2,note3} 
(the final classification is summarized in \cite[Table~1]{note3}; see also Table~\ref{table: SCTintComplQuad} below).
So, we also have a classification of the transformations as in the cases 
\eqref{quad1}, \eqref{quad2}, \eqref{caso3propC}, \eqref{caso3propDNcre}, and \eqref{caso3propDcre}.

The cases \eqref{nonquad1}, \eqref{nonquad2}, and \eqref{caso3propA} are 
 simpler to handle
(see the proofs of Corollary~\ref{cor main} 
and  Propositions~\ref{propAS0} and \ref{propASquart}),
and 
furthermore, under the assumption that the image $\ZZ$ is a smooth prime Fano variety,
 the transformations as in case \eqref{caso3propANcre}
have been classified in \cite{alzati-sierra}. 
It remains to consider the cases \eqref{caso3propBNcre} and \eqref{caso3propBcre}, 
which, as it also follows a posteriori from the examples, are the most intricate ones.

Recently, in \cite[Theorem~0.5]{note4} we classified the transformations as in  case \eqref{caso3propBcre},
and in the proof of \cite[Theorem~0.6]{note4} we studied the transformations 
as in a particular type of case \eqref{caso3propBNcre} with $(d_2,c)=(3,2)$,
that is \emph{cubo-cubic transformations into a cubic hypersurface}.
In Section~\ref{cubiche}  
we follow the approach used 
 in \cite{note4} to study 
more transformations as in case \eqref{caso3propBNcre}, namely when 
 $\ZZ$ is a factorial complete intersection.
This will allow us to obtain a complete effective classification of all special birational transformations 
into a factorial complete intersection whose base locus has dimension at most three.
\end{remark}

\begin{remark}\label{rem secants hypers}
Using the arguments of \cite[Proposition~2.3]{ein-shepherdbarron},
it is not difficult to show that 
if $\B\subset\PP^n$ is the base locus of a special birational transformation $\varphi:\PP^n\dashrightarrow\ZZ$
of type $(d_1,d_2)$ into a factorial variety $\ZZ$, then 
the $d_1$-\emph{th} secant variety of $\B$, $\mathrm{Sec}_{d_1}(\B)$,
\emph{i.e.} the closure of the union of all the $d_1$-secant lines of $\B$,
is a hypersurface in $\PP^n$ of degree $d_1 d_2 -1$; moreover,
if $p\in\mathrm{Sec}_{d_1}(\B)$ is a general point, then the 
union of all the $d_1$-secants of $\B$ through $p$ is an $(n-r'-1)$-dimensional linear space $L_p$ and  
 $L_p\cap\B\subset L_p$ in a hypersurface of degree $d_1$, where $r'$ denotes the dimension of the base locus of $\varphi^{-1}$.
However, we will not need this geometric property in the sequel, except in the case when $d_2=1$.
Indeed notice that
if $d_2=1$ then the above property implies that 
 $\mathrm{Sec}_{d_1}(\B)$ is the unique hypersurface of degree $d_1-1$ containing $\B$;
 in particular $h^0(\PP^n,\mathcal{I}_{\B}(d_1-1)) = 1$. 
\end{remark}

 \begin{remark} 
We recall here 
 a special case of 
the well-known relationship 
between 
the multidegree of a rational map between projective varieties
and the 
push-forward to projective space of the Segre class
of its base locus; see  
\cite[Proposition~4.4]{fulton-intersection},
\cite[p.~291]{crauder-katz-1989}, and
\cite[Subsection~2.3]{dolgachev-cremonas}.

Let $\varphi:\PP^n\dashrightarrow\ZZ$ be a special birational transformation of type $(d_1,d_2)$.
  Let $\B$ denote its base locus, $r=\dim (\B)$,  and 
    let $s_i(\N_{\B,\PP^n})\,  H_{\B}^{r-i}$ be the degree of 
  the $i$-th Segre class of the normal bundle of $\B$. 
  If $(\delta_0,\delta_1,\ldots,\delta_{n-1},\delta_n) = (1,d_1,\ldots,d_2\,\Delta,\Delta)$
  is the multidegree of $\varphi$, 
  then the following formula holds for $k=0,\ldots,n$:  
\begin{equation}\label{segresDegs}
\delta_{n-k} = d_1^{n-k}-  \binom{n-k}{r-k}\,d_1^{r-k}\,\deg(\B)  -\sum_{i=k}^{r-1}\binom{n-k}{i-k}\,d_1^{i-k}\,s_{r-i}(\N_{\B,\PP^n}) \,  H_{\B}^{i} .
     \end{equation} 
\end{remark}

The following general result 
puts restrictions on the Hilbert polynomial of the base locus of a special birational  transformation
(see also \cite[Proposition~4.4]{note} and \cite[Corollary~2.12]{vermeire}).
\begin{lemma}\label{cond hilbert pol}
  Let $\varphi:\PP^n\dashrightarrow\ZZ\subseteq\PP^N$ be 
  a special birational transformation of type $(d_1,d_2)$ 
  into a non-degenerate and linearly normal factorial variety $\ZZ$.
  Then we have
  \begin{equation}\label{formuleHilbPol}
\chi(\O_{\B}(d_1))=\binom{n+d_1}{d_1}-N-1\quad \mbox{and}\quad \chi(\O_{\B}(d_1-1))=\binom{n+d_1-1}{d_1-1}  + \Big\lceil \frac{d_2-1}{d_2} \Big\rceil -1.
\end{equation}
Moreover, if the dimension $r'$ of the base locus of $\varphi^{-1}$ is at most $n-3$, we also have:
 \begin{equation}\label{formuleHilbPol2}
 \chi(\O_{\B}(d_1-j))= \binom{n+d_1-j}{d_1-j}, \mbox{ whenever } 2 \leq j \leq n - r' -1.
\end{equation}
\end{lemma}
\begin{proof} 
Let the notation be
as in the beginning of the proof of Proposition~\ref{PropDim}.
If $V$ 
denotes
the 
$(N+1)$-dimensional 
 vector space associated to the linear system defining $\varphi$, then
 we have:  
\begin{align*}
 V &\subseteq H^0(\PP^n,\I_{\B}(d_1)) = H^0(\widetilde{\PP^n},d_1 H - E) = H^0(\widetilde{\PP^n},H') 
 = H^0(\ZZ,\pi'_{\ast}(\mathcal{O}_{\widetilde{\PP^n}})\otimes \mathcal{O}_{\ZZ}(1)) \\ & = H^0(\ZZ,\mathcal{O}_{\ZZ}(1)),
\end{align*}
where the last two equalities 
follow respectively 
 from the  projection formula 
 \cite[Exercise~II~5.1]{hartshorne-ag}
and from the 
  Zariski's Main Theorem \cite[Corollary~III~11.4]{hartshorne-ag}.
Now, from the non-degeneracy and linear normality of $\ZZ$,
 we get
$h^0(\ZZ,\mathcal{O}_{\ZZ}(1))=N+1$, and therefore
we 
have 
\begin{equation} \label{formula h0 d1}
h^0(\PP^n,\I_{\B}(d_1)) = N+1 .
              \end{equation}
Moreover, if $d_2>1$ then we clearly have $h^0(\PP^n,\mathcal{I}_{\B}(d_1-1)) = 0$, while
if $d_2=1$ 
then it follows from Remark~\ref{rem secants hypers}
that $h^0(\PP^n,\mathcal{I}_{\B}(d_1-1)) = 1$; thus, in formulas, we have
 \begin{equation}\label{formula h0 d1-1} 
  h^0(\PP^n,\mathcal{I}_{\B}(d_1-1)) = 1 - \Big\lceil \frac{d_2-1}{d_2} \Big\rceil .
 \end{equation}
Let now $r=\dim \B$ and $K=K_{\widetilde{\PP^n}}$. 
Using \eqref{changeBasis} and \eqref{canonicalBlowUp}, we can write for each $t\in\mathbb{Z}$:
\begin{equation*}
t H-E = K+(t H-E-K) 
= K  + (t - ((n-r) d_1 - n - 1) ) H + (n-r) H'.
\end{equation*}
For $t \geq (n-r) d_1 - n - 1$, we have that $(t - ((n-r) d_1 - n - 1) ) H$ is nef, 
$(n-r) H'$ is nef and big,  and 
therefore their sum is nef and big (see e.g. \cite[p.~20]{debarre}). Thus 
from the 
Kawamata-Viehweg vanishing 
theorem (see e.g. \cite[Corollary~4.13]{fujita-polarizedvarieties})
we deduce the following: (Let us notice that this 
follows from
\cite[Proposition~1]{bertram-ein-lazarsfeld} when $t\geq (n-r) d_1 - n$; see also the beginning of the proof of
\cite[Theorem~1.2]{bertram-ein-lazarsfeld}.)
\begin{equation}\label{formula: KV}
 h^i(\PP^n,\I_{\B}(t)) = h^i(\widetilde{\PP^n},t H-E) = 0,\mbox{ whenever } i > 0 \mbox{ and } t\geq (n-r) d_1 - n - 1 .
\end{equation}
Now, by \eqref{formulasDim} and \eqref{formula: KV}, using
 the exact sequence
$0\rightarrow \mathcal{I}_{\B} \rightarrow \mathcal{O}_{\PP^n}\rightarrow \O_{\B}\rightarrow 0 $, we obtain:
\begin{equation}\label{formula: KV2}
 \chi(\mathcal{O}_{\B}(t)) = h^0(\PP^n, \mathcal{O}_{\PP^n}(t)) - h^0(\PP^n,\mathcal{I}_{\B}(t)) , \mbox{ whenever }  n - r' \geq d_1 + 1 - t,
\end{equation}
and the assertion follows by \eqref{formula h0 d1} and \eqref{formula h0 d1-1}. 
 \end{proof}
  
 \section{Special cubic birational transformations} \label{cubiche}
 In this section we prove our main result, that is Theorem~\ref{thm main}. 
We first need to show that there are a finite number of possible cases to be analyzed.
 
 \begin{lemma}\label{numericalInvariants}
  Let $\varphi:\PP^6\dashrightarrow\ZZ\subseteq\PP^{6+a}$ be a special birational transformation 
  of type $(3,d)$  into a non-degenerate, linearly normal, factorial variety $\ZZ\subseteq\PP^{n+a}$ of degree $\Delta$ and codimension $a$. 
  Let $\B\subset\PP^6$ be the base locus of $\varphi$ 
  and assume that it is three-dimensional (which is automatic when $d\neq 1$).
  Let us define  $\epsilon(d)=1$ if $d=1$ and $\epsilon(d) = 0$ otherwise,
  and denote by $\lambda$ and $g$, respectively,
 the  degree and sectional genus of $\B$. Then the following hold:
 \begin{enumerate}
  \item The multidegree of $\varphi$ is given by $(1, 3, 9, -{\lambda}+27, -7 {\lambda}+2 g+79, {\Delta}\, d, {\Delta}) $.
  \item If $K_{\B}$ and $H_{\B}$ denote, respectively, a canonical divisor and a hyperplane section divisor of $\B$, then
  we have:
  \begin{align*}
  K_{\B} H_{\B}^2 &= -2\,{\lambda}+2\,g-2 , \quad 
  K_{\B}^2 H_{\B} = -39\,{\lambda}+14\,g+{\Delta}\,d-12\,a  + 12\,\epsilon(d) + 331 , \\
  K_{\B}^3 &= {\lambda}^2-77\,{\lambda}+14\,g-3\,{\Delta}\,d-{\Delta}-12\,a  + 60\,\epsilon(d)  +688 .
  \end{align*}
  \item  If $\mathcal{T}_{\B}$ denotes the tangent bundle of $\B$ then we have:
    \begin{align*}
    c_1(\mathcal{T}_{\B}) H_{\B}^2 &=  2\,{\lambda}-2\,g+2 , \quad  
    c_2(\mathcal{T}_{\B}) H_{\B} = -29\,{\lambda}+16\,g-{\Delta}\,d+227 , \\ 
    c_3(\mathcal{T}_{\B}) &= 230\,{\lambda}-102\,g+11\,{\Delta}\,d-{\Delta}-1842 .
    \end{align*}
  \item\label{surface} If $S\subset\PP^5$ denotes a smooth hyperplane section of $\B$, then we have:
  \begin{align*}
  K_S H_S &= -{\lambda}+2\,g-2,\quad 
  K_S^2 = -42\,{\lambda}+18\,g+{\Delta}\,d-12\,a  + 12\,\epsilon(d)  +327,  \\
  \chi(\mathcal{O}_S) &= -6\,\lambda + 3\,g - a + \epsilon(d) +  46, \quad 
   c_2(\mathcal{T}_S) =
 -30\,{\lambda}+18\,g-{\Delta}\,d+225 .
  \end{align*}
 \end{enumerate}
 \end{lemma}
\begin{proof}  
The proof is similar to that of \cite[Lemma~3.2]{note4}, but
we include it here for the sake of completeness.
From the Hirzebruch-Riemann-Roch formula \cite[Exercise~A~6.7]{hartshorne-ag}, we have:
\begin{align}
    K_{\B}\,  H_{\B}^2 &= -c_1(\T_{\B})\,  H_{\B}^2 = 2\,(g-1-H_{\B}^3), \\
 K_{\B}^2\,  H_{\B} &=  
12(\chi(\O_{\B}(H_{\B}))-\chi(\O_{\B})) -2\, H_{\B}^3 +3\, K_{\B}\,  H_{\B}^2  -c_2(\T_{\B})\,  H_{\B}.
\end{align}
Since $\B$ is embedded in $\PP^n$ with $n\leq 6$, we also have 
(see \cite[p.~543]{livorni-sommese-15}): 
\begin{equation}\label{doublePointFormula}
K_{\B}^3 = c_3(\T_{\B})+7\,c_2(\T_{\B})\,  H_{\B}-48\,\chi(\O_{\B})+(H_{\B}^3)^2-35\,H_{\B}^3-21\,K_{\B}\,  H_{\B}^2-7\,K_{\B}^2\,  H_{\B}.
\end{equation}
Moreover, from the exact sequence
$0\rightarrow\mathcal{T}_{\B}\rightarrow 
\mathcal{T}_{\PP^6}|_{\B}\rightarrow\mathcal{N}_{\B,\PP^6}\rightarrow0$
and since $s(\N_{\B,\PP^6})=c(\N_{\B,\PP^6})^{-1}$, we get:
\begin{align}
c_1(\T_{\B})\,H_{\B}^2 &= 7\,H_{\B}^3+s_1(\N_{\B,\PP^6})\, H_{\B}^2,\\ 
c_2(\T_{\B})\,H_{\B} &= 21\,H_{\B}^3+7\,s_1(\N_{\B,\PP^6})\, H_{\B}^2+s_2(\N_{\B,\PP^6})\, H_{\B}, \\
 c_3(\T_{\B}) &= 35\,H_{\B}^3+21\,s_1(\N_{\B,\PP^6})\, H_{\B}^2+7\,s_2(\N_{\B,\PP^6})\, H_{\B}+s_3(\N_{\B,\PP^6}) .
\end{align}
Since $\B$ is the base locus of a special birational transformation   
of type $(3,d)$, 
by \eqref{segresDegs} we have:
\begin{align}
 \Delta &=
 -540\,H_{\B}^3-135 \, s_1(\N_{\B,\PP^6})\, H_{\B}^2-18\,s_2(\N_{\B,\PP^6})\, H_{\B}  
        -s_3(\N_{\B,\PP^6})+729 , 
\\
  \Delta\,d &=  -90\,H_{\B}^3-15\,s_1(\N_{\B,\PP^6})\, H_{\B}^2-s_2(\N_{\B,\PP^6})\, H_{\B}+243.
\end{align}
Finally,  from the exact sequence
$0\rightarrow \T_S \rightarrow \T_{\B}|_{S}\rightarrow \O_{S}(H_S)\rightarrow 0 $,
we get:
\begin{equation}\label{EquazioneC2TS}
c_2(\T_S)=c_2(\T_{\B})\,H_{\B}+K_S\,H_S
=c_2(\T_{\B})\,H_{\B}+K_{\B}\,H_{\B}^2+H_{\B}^3 ,
\end{equation} 
and by \cite[Example~A~4.1.2]{hartshorne-ag} we also have:
\begin{equation}\label{EquazioneKS2}
K_{S}^2=12\,\chi(\O_S)-c_2(\T_S)=
12\,(\chi(\O_{\B})-\chi(\O_{\B}(-H_{\B})))-c_2(\T_S)  .
\end{equation}
Now, using  the formulas \eqref{formuleHilbPol}, the proof 
 is reduced to solving a system of linear equations.  
\end{proof}

\begin{lemma}\label{casifiniti} In the hypothesis and notation of Lemma \ref{numericalInvariants}, if $d\neq 1$ then
the possible $5$-tuples $(\lambda,g,\Delta,d,a)$
belong to a subset $\Gamma_{2480}^5\subset\mathbb{Z}^5$ of cardinality $2480$;
if $d=1$ then the possible $4$-tuples $(\lambda,g,\Delta,a)$
belong to a subset $\Gamma_{1139}^4\subset\mathbb{Z}^4$ of cardinality $1139$.
\end{lemma}
\begin{proof}
 We have $3\leq \lambda \leq 27$ as $\B$ is a non-degenerate variety cut out by cubics and 
 of codimension $3$. Moreover, $g$ is non-negative and limited from above by
 the Castelnuovo bound \cite{CastelnuovoBound}:  
 \begin{equation}
 g \leq \Bigl\lfloor {\frac{\lambda-2}{3} }\Bigr\rfloor\left( \lambda -4 - \left(\Bigl\lfloor \frac{\lambda -2}{3} \Bigr\rfloor - 1 \right)\, \frac{3}{2} \right) .
 \end{equation}
 From 
 Proposition~\ref{PropDim}
 we obtain that $d\leq 5$;
 moreover, $d = 5$ if and only if $\Delta = 1$,
 and 
 $d = 4$ if and only if $\Delta = 2$. Of course we also have that
 if $\Delta=1$ then $a=0$, if $\Delta = 2$ then $a=1$, and if $\Delta\geq 3$ then $a\geq 1$.
 Furthermore, 
  the multidegree $(\delta_0,\ldots,\delta_6)$  of $\varphi$ 
  must satisfy a series of inequalities (see \cite[Corollary~1.6.3]{lazarfeld-positivity} and \cite[Subsection~1.4]{dolgachev-cremonas}),
   some of which are the following: 
   $\delta_1\,\delta_3\geq \delta_4$,
   $\delta_1\,\delta_4\geq \delta_5$,
   $\delta_3^2\geq \delta_2\,\delta_4$,
   $\delta_4^2\geq \delta_3\,\delta_5$,
   $\delta_5^2\geq \delta_4\,\delta_6$. 
   Using Lemma~\ref{numericalInvariants}, these five inequalities become:
   \begin{gather}
    2\,{\lambda}-g+1\geq 0, \label{dis1}
    \\ -21\,{\lambda}+6\,g-\Delta\,d+237\geq 0, \label{dis2}\\
    {\lambda}^2+9\,{\lambda}-18\,g+18\geq 0, \label{dis3}\\
    49\,{\lambda}^2-28\,{\lambda}\,g+4\,g^2+(\Delta\,d-1106)\,{\lambda}+316\,g-27\,\Delta\,d+6241\geq 0, \label{dis4}\\ 
    7\,\Delta\,{\lambda}-2\,\Delta\,g+\Delta^2\,d^2-79\,\Delta \geq 0. \label{dis5}
   \end{gather}
 Finally from \cite[Theorem~0.7.3]{livorni-sommese-15} and Lemma~\ref{numericalInvariants},
   we deduce the following four inequalities involving  functions of $\lambda, g, \Delta, d, a$:
   \begin{gather}
    {\lambda}^2+7\,{\lambda}-10\,g-2\,{\Delta}\,d+12\,a -12\,\epsilon(d)   -92 \geq 0,\\
 -10\,{\lambda}+8\,g+2\,{\Delta}\,d-12\,a  +12\,\epsilon(d)   +94 \geq 0 , \\
 -290\,{\lambda}+140\,g-13\,{\Delta}\,d+{\Delta}+2290 \geq 0 ,\\
-147\,{\lambda}+74\,g+12\,{\Delta}\,d-{\Delta}-84\,a  +108\,\epsilon(d)   +1183 \geq 0 .
   \end{gather}
   All these conditions together define a
 subset of $\mathbb{Z}^5$ consisting of  $3619$ elements, of which $2480$ have $d$-coordinate different from $1$.
\end{proof}

\begin{theorem}\label{thm main}
 Let $\varphi:\PP^n\dashrightarrow\ZZ\subseteq \PP^{n+a}$ be a special  birational transformation 
 of type $(3,d)$, $d>1$, into a non-degenerate factorial
  complete intersection.
If the base locus $\B$ of $\varphi$ has dimension $3$, then $n=6$ and one of the following cases holds:
 \begin{itemize}
\item $a=0$, $\varphi$ is a cubo-quintic Cremona transformation, and we have two cases according to
 \cite[Theorem~0.5]{note4}:
\begin{itemize}
 \item $\B$ is a threefold of degree $14$, sectional genus $15$, with trivial canonical bundle,
 and given by the Pfaffians of a skew-symmetric matrix;
 \item  $\B$ is a threefold of degree $13$ and sectional genus $12$,
 which admits the structure of a conic bundle over $\PP^2$.
\end{itemize}
 \item $a=1$, 
 $\B$ is a threefold  of degree $12$ and sectional genus $10$,
 and two cases occur:
 \begin{itemize}
  \item $\varphi$ is a cubo-cubic transformation into a cubic hypersurface and 
  $\B$ admits the structure of a fibration over $\PP^1$ whose generic fiber is a sextic del Pezzo surface;
  \item $\varphi$ is a cubo-quartic transformation into a quadric hypersurface and
  $\B$ is the blow-up at one point of a threefold of degree $13$ and sectional genus $10$
  which admits the structure of a
  fibration over $\PP^1$ whose generic fiber is a quintic del Pezzo surface.
  \end{itemize}
 \item $a=2$, $\varphi$ is a cubo-cubic transformation into a complete intersection of two quadrics in $\PP^8$ and
 $\B$ is a threefold of degree $11$ and sectional genus $8$ obtained by blowing-up $3$ points on 
 a  $3$-dimensional linear section of $\mathbb{G}(1,5)\subset\PP^{14}$.
 \item $a=3$, $\varphi$ is a cubo-quadric transformation into a complete intersection of three quadrics in $\PP^9$ and
 $\B$  is a threefold of degree $10$ and sectional genus $6$,
 which admits the structure of a scroll over  $\PP^2$ with four double points blown up. 
\end{itemize}
\end{theorem}
\begin{proof}
The fact that $n=6$ follows by \eqref{formulasDim}. Thus we are in the situation of Lemmas~\ref{numericalInvariants} and \ref{casifiniti}.

   If $\ZZ$ is a complete intersection of type $(e_1,\ldots,e_a)$, 
   then $\Delta=\Pi_{i=1}^a e_i$ and $c = \Sigma_{i=1}^a e_i - a$.
   From 
   Proposition~\ref{PropDim}
   we see that $a = \Sigma_{i=1}^a e_i + d - 5$ and in particular this forces $a+d\leq 5$.
   By imposing these conditions on the set $\Gamma_{2480}^5$ 
   given in Lemma~\ref{casifiniti}, we get
   a set $\Gamma_{174}^5$ consisting of $174$ not excluded $5$-tuples
    $(\lambda,g,\Delta,d,a)$, and of which the not excluded triples 
 $(\Delta,d,a)$ are the following:      
      $(1,5,0)$, $(4,2,1)$, $(3,3,1)$, $(2,4,1)$, $(6,2,2)$, $(4,3,2)$, $(8,2,3)$;
      moreover, this forces $ 7 \leq \lambda \leq 18 $ and 
      $0\leq g \leq 28$.
   
In the following we apply general results from adjunction theory 
for which we refer mainly to
\cite{fujita-polarizedvarieties,
beltrametti-sommese,
ionescu-smallinvariants,
ionescu-adjunction,
sommese-adjunction-theoretic,
sommese-adjunction-mapping,
beltrametti-biancofiore-sommese-loggeneral}.

Consider the complete linear system $|K_{\B} + 2 H_{\B}|$ on $\B$. 
By \cite{sommesseDuke,sommese-adjunction-mapping}, it is base-point free
unless $(\B,H_{\B})$ is one of the following:
\begin{enumerate}
 \item\label{a1} $(\PP^3,\O_{\PP^3}(1))$;
 \item\label{a2} a quadric $(Q^3,\O_{Q^3}(1))$;
 \item\label{a3} a scroll over a curve. 
\end{enumerate}
Of course, cases \eqref{a1} and \eqref{a2} do not occur. In case \eqref{a3} we have 
two equations:
\begin{equation}
 (K_{\B} + 3 H_{\B})^3 = (K_{\B} + 3 H_{\B})^2\, H_{\B} = 0,
\end{equation}
which, using Lemma~\ref{numericalInvariants}, translate into two polynomial equations in the variables 
$(\lambda,g,\Delta,d,a)$ without common solutions in the set $\Gamma_{174}^5$. 
So we deduce that $|K_{\B} + 2 H_{\B}|$ is base-point free. 

Now we show that $K_{\B} + 2 H_{\B}$ is nef and big with the exception of one case.
By \cite{sommese-adjunction-theoretic,sommese-adjunction-mapping}, we have that $K_{\B} + 2 H_{\B}$ is nef and big unless 
$(\B,H_{\B})$ has the structure of one of the following:
 \begin{enumerate}
  \item\label{b1}  a del Pezzo variety;
  \item\label{b2}  a quadric fibration over a smooth curve;
  \item\label{b3} a scroll over a smooth surface.
 \end{enumerate}
In case \eqref{b1} we get three equations:
\begin{equation}\label{delpezzoeq}
 K_{\B} H_{\B}^2 + 2 H_{\B}^3 = K_{\B}^2 H_{\B} + 2 K_{\B} H_{\B}^2 = K_{\B}^3 + 2 K_{\B}^2 H_{\B} = 0.
\end{equation}
In case \eqref{b2} we get two equations:
\begin{equation}\label{fibrazioneeq}
 (K_{\B} + 2 H_{\B})^3 = (K_{\B} + 2 H_{\B})^2 H_{\B} = 0 .
\end{equation}
In case \eqref{b3} we get one equation:
\begin{equation}\label{scrolleq}
 (K_{\B} + 2 H_{\B})^3 =  0 .
\end{equation}
Using Lemma~\ref{numericalInvariants} all these equations translate into polynomial equations in 
$(\lambda,g,\Delta,d,a)$.
    Then one can see that the systems
    \eqref{delpezzoeq} and \eqref{fibrazioneeq} have no solutions in $\Gamma_{174}^5$,
    while \eqref{scrolleq} has exactly the two solutions:
      $(10,6,8,2,3)$ and $(12,11,8,2,3)$.
Let us recall that in case \eqref{b3} we have
 $\B=\PP_Y(\mathcal{E})$ and the numerical invariants of 
 the polarized surface $(Y,\det \E)$ can 
 be determined in terms of those of $(\B,H_{\B})$, see \cite[Lemma~5.2]{besana-biancofiore-numerical}.
Thus we get two cases:
\begin{itemize}
 \item $(\lambda,g,\Delta,d,a) = (10,6,8,2,3)$,  $\chi(\mathcal{O}_Y)= 1$, $K_Y^2= 5$, $c_1(\mathcal{E})^2 =20$, and $c_2(\mathcal{E})= 10$;
 \item $(\lambda,g,\Delta,d,a) = (12,11,8,2,3)$, $\chi(\mathcal{O}_Y)= 4$, $K_Y^2=18$, $c_1(\mathcal{E})^2 =29$, and $c_2(\mathcal{E})=17$.
\end{itemize}
The first case is possible and corresponds to 
a 
cubo-quadric transformation into a complete intersection of three quadrics,
see  Example~\ref{cubo quadric into three quadrics}.
Instead, the second case does not occur
since 
$Y$ must be non-ruled and hence we would have $c_1(\mathcal{E})^2 \leq 2 g-2$ (see also \cite[Proposition~5.3]{threefoldsdegree12}).

Thus, from now on we can assume that $K_{\B}+2 H_{\B}$ is nef and big.
It follows that $(\B,H_{\B})$ admits a unique (minimal) reduction, that is there exists a 
pair $(\R,H_{\R})$, where $\R$ is a smooth irreducible threefold 
and $H_{\R}$ an ample divisor on $\R$, such that 
there is a morphism $\pi:\B\to \R$ expressing $\B$ 
as the blowing-up of $\R$ at a finite number $\nu$ of distinct points $p_1,\ldots,p_\nu$
and $H_{\B} \simeq \pi^{\ast} H_{\R}-\Sigma_{i=1}^\nu \pi^{-1}(p_i)$. Moreover,
$K_{\R} + 2 H_{\R}$ is very ample, where
$K_{\R}$ denotes the canonical divisor of $\R$.

We have 
the following upper bound for the number $\nu$ of exceptional divisors on $(\B,H_{\B}$):\footnote{Notice that 
 the difference between the right and left side of the inequality \eqref{formula4secants} coincides 
 with the sum $\Sigma_{l}\binom{5+l^2}{4}$, where
 $l$ runs over all lines contained in the surface $S\subset\PP^5$  
 having self-intersection $\leq -2$. Thus \eqref{formula4secants} is a strict inequality if and only if 
 $S$ contains a line with self-intersection $\leq -6$.
 }
 \begin{align}
 \nonumber \nu \leq& (1/8)\,{\lambda}^4-(1/2)\,{\lambda}^2\,{\Delta}\,{d}+(1/2)\,{\Delta}^2\,{d}^2+(3/4)\,{\lambda}^3-3\,{\lambda}^2\,{g}+(1/2)\,{\lambda}\,{\Delta}\,{d} \\
\label{formula4secants} & +5\,{g}\,{\Delta}\,{d}+3\,{\lambda}^2\,{a}-6\,{\Delta}\,{d}\,{a}-(433/8)\,{\lambda}^2+20\,{\lambda}\,{g}+13\,{g}^2+(25/2)\,{\Delta}\,{d} \\
\nonumber & -7\,{\lambda}\,{a}-30\,{g}\,{a}+18\,{a}^2+(2825/4)\,{\lambda}-98\,{g}-21\,{a}-2969 .
  \end{align}
Indeed the inequality \eqref{formula4secants} is obtained by applying 
to the general hyperplane section $S$ of $\B$ the formula, due to P. Le Barz (see \cite{barz-quadrisecantes,barz-multisecantes}),
calculating the number of $4$-secant lines of a surface in $\PP^5$,
and taking into account that, since $S$ is cut out by cubics, it cannot have $4$-secant lines;
see also \cite[Subsection~2.4]{note4}. 
In particular, it follows that there are a finite number of not excluded $6$-tuples $(\lambda,g,\nu,\Delta,d,a)$,
which actually are $4237$.
Let $\Gamma_{4237}^6$ denote the set of all these $6$-tuples.
In this set we have $ 11 \leq \lambda\leq 18$, $7\leq g\leq 28$, and $0\leq \nu\leq 181$.

Now we show that $K_{\R} + H_{\R}$ is nef and big in exactly one case.  
For $j=0,\ldots,3$ we define the $j$-th pluridegree of $(\R,H_{\R})$  as
$d_j(\R)={(K_{\R}+H_{\R})}^j\,{H_{\R}}^{3-j}$.
Notice that we have
$K_{\R}\,H_{\R}^2=K_{\B}\,H_{\B}^2-2\nu$, $K_{\R}^2\,H_{\R}=K_{\B}^2\,H_{\B}+4\nu$, 
and
$K_{\R}^3=K_{\B}^3-8\nu$, so that
from Lemma~\ref{numericalInvariants} we can express the pluridegrees 
of $(\R,H_{\R})$ as polynomial functions of $(\lambda,g,\nu,\Delta,d,a)$.
Now,
if $K_{\R} + H_{\R}$ is nef and big, we have the following inequalities (see \cite{beltrametti-biancofiore-sommese-loggeneral}):
\begin{equation}
 d_1(\R) - 1\geq0,\quad  d_2(\R) - 1\geq0,\quad  d_3(\R) - 1\geq0,\quad
 d_1(\R)^2 - d_{2}(\R) d_{0}(\R) \geq 0,
\end{equation}
that hence become respectively:
\begin{gather}
\label{c1}  -{\lambda}+2\,g-{\nu}-3\geq 0, \\ 
\label{c2} {\Delta}\,d-42\,{\lambda}+18\,g+{\nu}-12\,a+326\geq 0, \\ 
\label{c3} {\lambda}^2-199\,{\lambda}+62\,g-{\nu}-{\Delta}-48\,a+1674\geq 0, \\ 
\label{c4} -{\lambda}\,{\Delta}\,d-{\nu}\,{\Delta}\,d+43\,{\lambda}^2-22\,{\lambda}\,g+4\,g^2+43\,{\lambda}\,{\nu}-22\,g\,{\nu} \\ \nonumber \quad \quad
+12\,{\lambda}\,a+12\,{\nu}\,a-323\,{\lambda}-8\,g-323\,{\nu}+4 \geq 0 .
\end{gather}
The $6$-tuples $(\lambda,g,\nu,\Delta,d,a)$, belonging to the set $\Gamma_{4237}^6$ 
that satisfy \eqref{c1}, \eqref{c2}, \eqref{c3}, and \eqref{c4} are the two following:
$(14,15,0,1,5,0)$ and $(18,28,0,3,3,1)$. The last one does not occur:
indeed we have $d_1(\R)^2-d_2(\R) d_0(\R)=0$ and $d_2(\R)^2-d_3(\R) d_1(\R)=1521\neq 0$, and this is a contradiction by
 \cite[Lemma~1.1, (1.1.2)]{beltrametti-biancofiore-sommese-loggeneral} (see also the proof of \cite[Theorem~0.6]{note4}).
 The tuple $(14,15,0,1,5,0)$ corresponds to a Cremona transformation 
 and we can apply \cite[Theorem~0.5]{note4},
 see also Example~\ref{example Cremona 14-15}.

 We now can assume that $K_{\R} + H_{\R}$ is not nef and big. Then by \cite{sommese-adjunction-theoretic} we have that
 $(\R,H_{\R})$ has the structure of one of the following:
  \begin{enumerate}
  \item\label{e1} $(\PP^3,\O_{\PP^3}(3))$;
  \item\label{e2} $(Q^3,\O_{Q^3}(2))$;
  \item\label{e3} a Veronese fibration over a curve;
  \item\label{e4} a Mukai variety;
  \item\label{e5} a del Pezzo fibration over a curve;
  \item\label{e6} a conic bundle over a surface.
  \end{enumerate}
One see easily that 
cases \eqref{e1} and \eqref{e2} do not occur (by applying for instance \cite[Lemma~3.3]{besana-biancofiore-numerical}).
In case \eqref{e3} we have two equations:
\begin{equation}
 (2K_{\R}+3H_{\R})^3 = (2K_{\R}+3H_{\R})^2\,H_{\R}   = 0,
\end{equation}
which admit no solutions in $\Gamma_{4237}^6$.
In case \eqref{e4} we have three equations:
\begin{equation}
 K_{\R}^3+H_{\R}^3 = 
 K_{\R}^2\,H_{\R}-H_{\R}^3 = 
 K_{\R}\,H_{\R}^2+H_{\R}^3 = 0,
\end{equation}
which admit in $\Gamma_{4237}^6$ the unique solution $(11,8,3,4,3,2)$,
corresponding to a cubo-cubic transformation into a complete intersection of two quadrics, see Example~\ref{example G-1-5}.
In case \eqref{e5} we have two equations and an inequality:
\begin{equation}\label{delPez0}
d_3(\R) = 0,\quad  d_2(\R) = 0, \quad \mbox{ and } \quad d_1(\R) - 1 \geq 0 .
\end{equation}
The tuples in the set $\Gamma_{4237}^6$ that satisfy \eqref{delPez0} are the two following: $(12,10,0,3,3,1)$ and
$(12,10,1,2,4,1)$.
Moreover one has
that in these two cases $(\R,H_{\R})$ is a fibration over $\PP^1$ 
whose generic fibre is a del Pezzo surface of degree $6$ and $5$, respectively
(see \cite[p.~13]{besana-biancofiore-numerical}; see also \cite[Proposition~5.16]{threefoldsdegree12}).
Both these cases are possible, see Examples~\ref{cubo-cubic into cubic} and \ref{cubo-quartic into quadric}.
Finally, in case \eqref{e6} we have one equation and 
two inequalities:  
\begin{equation}\label{coninBunl0}
 d_3(\R) = 0, \quad d_2(\R) - 1 \geq 0,\quad  \mbox{ and }\quad   d_1(\R)^2 - d_2(\R)\,d_0(\R) \geq 0  . 
\end{equation}
There is only one tuple in $\Gamma_{4237}^6$ that satisfy \eqref{coninBunl0}, which is $(13,12,0,1,5,0)$.
This  corresponds to a Cremona transformation 
and we can apply \cite[Theorem~0.5]{note4},
 see also Example~\ref{example Cremona 13-12}.
This conclude the proof.
\end{proof}

\begin{remark}
 The proof 
 of Theorem~\ref{thm main}
 can be easily translated into a computer program. 
 Our code written for {\sc Macaulay2} \cite{macaulay2}
 is available as an ancillary file
 included in the arXiv submission.
 Although our code is not optimal, 
  it executes all of the proof in a few seconds.
\end{remark}

\begin{corollary}\label{cor main}
Table~\ref{table: SCTintCompl} classifies all
the special birational transformations of type $(3,d)$, $d>1$, 
into a 
factorial complete intersection 
and whose base locus has dimension at most three. 
\end{corollary}
\begin{proof}
 If the dimension of the base locus $\B$ is three, this is the content of  Theorem~\ref{thm main}.
 In the case when the dimension of $\B$ is two, from Corollary~\ref{casiProp} we see
 that the transformation must be a cubo-quadric Cremona transformation, and then we apply \cite{crauder-katz-1989}.
If the dimension of $\B$ is one, from Corollary~\ref{casiProp} we have 
either a cubo-cubic Cremona transformation or
a cubo-quadric transformation into a quadric.
The first case is classical (cf. \cite{katz-cubo-cubic}).
In the second case we apply Lemma~\ref{cond hilbert pol};
and to get an example we can take
 the restriction to a hyperplane of
a special cubo-quadric Cremona transformation of $\PP^4$.
\end{proof}
\begin{table}[htbp]
\centering
\tabcolsep=6.8pt 
\begin{tabular}{llllllclll}
 \hline 
  & $r$  & $n$ & $a$ & $\lambda$ & $g$ & Abstract structure of $\B$ & d & $\Delta$  & Existence \\  
 \hline 
 \rowcolor{Gray}
 I & $1$ & $3$ & $0$ & $6$ & $3$ & Determinantal curve &  $3$ & $1$ &   \cite{katz-cubo-cubic}  \\ 
 II & $1$ & $3$ & $1$ & $5$ & $1$ & Elliptic curve &  $2$ & $2$ &  \cite{crauder-katz-1989}   \\ 
\rowcolor{Gray}
III & $2$ & $4$ & $0$ & $5$ & $1$ & Elliptic scroll &  $2$ & $1$ &   \cite{crauder-katz-1989}  \\ 
IV & $3$ & $6$ & $0$ & $14$ & $15$ & Pfaffian threefold &  $5$ & $1$ & Example~\ref{example Cremona 14-15}    \\ 
\rowcolor{Gray}
V & $3$ & $6$ & $0$ & $13$ & $12$ & Conic bundle over $\PP^2$ & $5$ & $1$ & Example~\ref{example Cremona 13-12}    \\
VI & $3$  & $6$ & $1$ & $12$ & $10$ & Sextic del Pezzo fibration over $\PP^1$ & $3$ & $3$ & Example~\ref{cubo-cubic into cubic} \\
\rowcolor{Gray}
VII & $3$  & $6$ & $1$ & $12$ & $10$ & \begin{tabular}{c} Blow up at one point of a quintic \\ del Pezzo fibration over $\PP^1$\end{tabular} & $4$ & $2$ & Example~\ref{cubo-quartic into quadric} \\
VIII  & $3$  & $6$ & $2$ & $11$ & $8$  & \begin{tabular}{c} Blow up at three points \\ of $\mathbb{G}(1,5)\cap\PP^9\subset\PP^9$ \end{tabular} & $3$ & $4$ & Example~\ref{example G-1-5}  \\
\rowcolor{Gray}
IX & $3$  & $6$ & $3$ & $10$ & $6$  & \begin{tabular}{c} Scroll over $\PP^2$ with four \\ double points blown up \end{tabular} & $2$ & $8$ & Example~\ref{cubo quadric into three quadrics} \\
\end{tabular}
 \caption{All the types of special birational transformations $\PP^n\dashrightarrow\ZZ\subseteq\PP^{n+a}$ 
 of type $(3,d)$, with $d>1$ and base locus $\B$ of dimension $r\leq 3$, degree $\lambda$, sectional genus $g$,
 into a factorial complete intersection $\ZZ\subseteq\PP^{n+a}$ of degree $\Delta$. } 
\label{table: SCTintCompl} 
\end{table}

\begin{remark}\label{rem Fano}
Let $\varphi:\PP^n\dashrightarrow\ZZ\subseteq\PP^{n+a}$ be a special 
birational transformation 
 of type $(3,d)$, with $d>1$ and base locus  of dimension three,
 into a prime Fano manifold $\ZZ$.
 If $\ZZ$ is not a complete intersection, then 
 $\varphi$ is a cubo-cubic transformation into $\mathbb{G}(1,4)\subset\PP^9$ and
    its base locus 
    is  a threefold of degree $10$ and sectional genus $6$ 
     with the structure of a scroll over $\PP^2$.
 \begin{proof}
  Indeed,
 by Corollary~\ref{casiProp} we see 
 that $n=6$ and the coindex $c$ of $\ZZ$ satisfies $c=5-d\leq 3$.
 Thus such manifolds $\ZZ$ are completely classified 
  \cite{fujita-polarizedvarieties,mukai-biregularclassification}
  and we
  get  conditions on the possible 
 $\Delta$ and $a$. 
    By imposing these conditions on the set $\Gamma_{2480}^5 \setminus \Gamma_{174}^5$ 
   (see Lemma~\ref{casifiniti} and the proof of Theorem~\ref{thm main}), 
   we get
  a set $\Gamma_{88}^5$ consisting of $88$ not excluded $5$-tuples
    $(\lambda,g,\Delta,d,a)$.
    Then, by proceeding as in the proof of Theorem~\ref{thm main},
    one  gets that 
    $\varphi$ is a cubo-cubic transformation into $\mathbb{G}(1,4)\subset\PP^9$ and
    its base locus 
    is as in one of the two following cases:
    \begin{itemize}
     \item  a threefold of degree $10$ and sectional genus $6$ 
     with the structure of a scroll over $\PP^2$;
     \item  a minimal threefold of log-general type of degree $16$ 
     and sectional genus $23$.
    \end{itemize}
The first case actually occurs: the map defined
by the maximal minors of a generic $3\times 5$ matrix of linear forms on $\PP^6$
gives us an example (see also \cite[Subsection~4.4.1]{note4}).   
The second case does not occur. Indeed
such a threefold must be linked 
inside the complete intersection of three cubics 
to a smooth threefold $X\subset \PP^6$
of degree $11$ and sectional genus $13$ (see e.g. \cite[p.~423]{okonek}).
By Lemma~\ref{cond hilbert pol} and \cite[Proposition~4.7]{Hartshorne1994}, we can determine 
the Hilbert polynomial of $X$, so that in particular we deduce 
that $\chi(\mathcal{O}_{X})-\chi(\mathcal{O}_{X}(-1)) = 8$. 
But this is a contradiction by the maximal list of invariants of threefolds of degree $11$
obtained in
 \cite{beltrametti_schneider_sommese_1992,besana-biancofiore-deg11}.
 \end{proof}
 \end{remark}

\begin{remark}\label{cubicP7}
Let $\varphi:\PP^n\dashrightarrow\ZZ$ be a special 
birational transformation 
 of type $(3,d)$, with $d>1$ and base locus  of dimension four,
 into a factorial variety $\ZZ$.
From Proposition~\ref{PropDim}, it follows that
 $n=7$, the base locus of $\varphi^{-1}$ has dimension four, and
 one of the following holds:
 \begin{itemize}
  \item $\varphi$ is a cubo-cubic Cremona transformation;
  \item $\varphi$ is a cubo-quadric transformation into a variety $\ZZ$ of coindex $2$.
 \end{itemize}
To the best of the author's knowledge, 
no examples are known of such transformations.
Notice that 
they look like extensions
to $\PP^7$
of special cubic transformations of $\PP^6$. 
In \cite[Theorem~0.6]{note4}, we showed that the base locus of a  
special Cremona transformation of $\PP^7$ must necessarily be a fourfold 
of degree $12$ and sectional genus $10$
 with a structure of
 sextic del Pezzo fibration over $\PP^1$.
\end{remark} 
 \subsection{Cubo-linear birational transformations}
 In this subsection,
we describe the special cubic birational transformations 
 into a factorial complete intersection and whose base locus has dimension at most three
 that are not considered in Corollary~\ref{cor main} and Table~\ref{table: SCTintCompl}.
The following result tells us that there are no relevant transformations.
 
 \begin{theorem}\label{thm cubo-linear}
 Let $\varphi:\PP^n\dashrightarrow\ZZ\subseteq \PP^{n+a}$ be a special  birational transformation 
 of type $(3,1)$ into a  factorial
  complete intersection and whose base locus $\B$ has dimension $r\leq3$.
  Then $n\in\{3,4,5\}$, $\B$ has codimension $2$, degree $5$, and sectional genus $2$,
  and $\ZZ\subset\PP^{n+2}$ is a complete intersection of two quadrics.
  Moreover, if $n=4$ then $\B$ is 
the image of $\PP^2$ via the linear system of quartic curves with $7$ simple base points and one double point;
if $n=5$ then $\B$  admits the structure of quadric fibration over $\PP^1$.
\end{theorem}

\begin{proof}
From Proposition~\ref{PropDim} we deduce that there are two cases:
\begin{enumerate}
 \item\label{case1CuboLinear} $(r,n)\in\{(1,3),(2,4),(3,5)\}$, $r'=1$, and $c=2$, so that $(\Delta,a)\in\{(4,2),(3,1)\}$;
 \item\label{case2CuboLinear} $r=3$, $n=6$, $r'=4$, and $c=4$.
\end{enumerate}

\noindent
\emph{Case \eqref{case1CuboLinear}.} Since $r'=1$,
by Lemma~\ref{cond hilbert pol}  we have $n-1=r+1$ conditions on the Hilbert polynomial of $\B$,
and hence we can write it as a function of $a$. 
Thus we see that
$\lambda=7-a$ and $g=6-2 a$, and this leaves two cases: 
either $(\lambda,g,\Delta,a)=(5,2,4,2)$ 
or
$(\lambda,g,\Delta,a)=(6,4,3,1)$.
In the former case, we can apply \cite{ionescu-smallinvariants}.
In the latter case, $\B$ must be a complete intersection 
of a quadric and a cubic, so that  $\ZZ$ is a cubic hypersurface with a double point $p$ and 
 $\varphi^{-1}$ is the projection from $p$. This case is excluded 
 because we have that the singular locus of $\ZZ$ coincides with the base locus of $\varphi^{-1}$.
 
\noindent
\emph{Case \eqref{case2CuboLinear}.}
 Let the notation be as in Lemmas~\ref{numericalInvariants} and \ref{casifiniti}.
 If $\ZZ$ is a complete intersection of type $(e_1,\ldots,e_a)$, 
   then $\Pi_{i=1}^a e_i = \Delta$ and $\Sigma_{i=1}^a e_i = a + c = a + 4$.
   Thus we have that $(\Delta,a)\in\{(5, 1), (8, 2), (9, 2), (12, 3), (16, 4)\}$, and 
   the only $4$-tuples $(\lambda,g,\Delta,a)$ in the set  $\Gamma_{1139}^4$ 
   that satisfy this are:
$(15, 21, 16,  4)$, $(15, 20, 16, 4)$, $(14, 17, 16, 4)$, 
$(13, 14, 16, 4)$, 
$(15, 19, 12,  3)$, 
$(18, 28, 9,  2)$.
Now, we can repeat step by step the proof of Theorem~\ref{thm main} 
in order to conclude that the first three cases do not occur, 
while in the other three ones we have that
$\B$ is a minimal threefold of log-general type.
In these three remaining cases, the multidegree of $\varphi$  
is respectively: $(1, 3, 9, 14, 16, 16, 16)$,
$(1, 3, 9, 12, 12, 12, 12)$, and
$(1, 3, 9, 9, 9, 9, 9)$, from which follows that  
the dimension $r'$ of the base locus of $\varphi^{-1}$ is respectively $3$, $2$, and $1$.
This is a contradiction since we must have $r'=4$ (see also Remark~\ref{rem examples}).
\end{proof}
The following  example shows that all the cases
of Theorem~\ref{thm cubo-linear} actually occur.

 \begin{example}\label{example 2.9}
 Let $P\subset\PP^n$, $n\in\{3,4,5\}$, be a two-codimensional  subspace, and let 
 $W\subset\PP^n$ be a generic complete intersection of a quadric and a cubic containing $P$.
 Then the residual intersection 
  $X=\overline{W\setminus P}$ is a two-codimensional smooth irreducible variety
  of degree $5$ and sectional genus $2$, cut out by one quadric and two cubics.
   The linear system of cubics through $X$ defines a birational map 
  $\PP^n\dashrightarrow\ZZ\subset\PP^{n+2}$, where $\ZZ$ is a smooth
  complete intersection of two quadrics.
 \end{example}

  \begin{remark}\label{rem examples}
Here we give three examples of birational maps of type $(3,1)$ with smooth and irreducible base locus 
from $\PP^6$ 
into a complete intersection, having multidegrees 
as in the three cases considered in the final part of the proof of Theorem~\ref{thm cubo-linear}.
These maps are not special since the singular locus of the image 
coincides set-theoretically with the base locus of the inverse map.

\paragraph{(1)}
Let $X\subset\PP^6$ be a generic complete intersection of 
 one quadric and two cubics. Then $X$ is a smooth irreducible threefold 
 of degree $18$ and sectional genus $28$, and
 the linear system of cubics through $X$
 defines a birational map into a 
 complete intersection $\ZZ$ of two cubics with $\mathrm{Sing}(\ZZ)=1$.
\paragraph{(2)} Let $Y\subset\PP^4\subset\PP^6$ be a generic cubic hypersurface. Let $W\subset\PP^6$ be 
 a generic complete intersection of one quadric and two cubics containing $Y$. Then 
 the residual intersection $X=\overline{W\setminus Y}$ is a smooth irreducible threefold 
 of degree $15$ and sectional genus $19$, cut out by one quadric and three cubics. The linear system 
 of cubics through $X$ defines a  birational map 
 into a complete intersection $\ZZ$ of two quadrics and one cubic with  $\mathrm{Sing}(\ZZ)=2$. 
 \paragraph{(3)}
 Let $Y\subset \PP^5\subset \PP^6$ be a generic complete intersection of two quadrics.
 Let $W\subset\PP^6$ be 
 a generic complete intersection of one quadric and two cubics containing $Y$.
 Then the residual intersection 
 $Y' = \overline{W\setminus Y}$  is a smooth irreducible threefold 
 of degree $14$ and sectional genus $16$, cut out by one quadric and three cubics.
 The linear system of cubics through $Y'$ defines a birational map 
 into a $6$-fold $Z\subset \PP^9$ of degree $13$ and sectional genus $14$, cut out by one quadric and four cubics.
 Let $X\subset\PP^6$ be a generic three-dimensional linear section of $Z$.
 Then the linear system of cubics through $X$ defines a  birational map $\PP^6\dashrightarrow\ZZ\subset\PP^{10}$
 into a complete intersection $\ZZ$ of four quadrics 
 with $\mathrm{Sing}(\ZZ)=3$.
\end{remark}

\section{New and revised examples of special cubic transformations}\label{Examples}
In this section we give explicit constructions for
all the examples of special cubic birational transformations of $\PP^6$ 
as in Table~\ref{table: SCTintCompl}.

Most of the calculations in the examples below 
are done using {\sc Macaulay2} \cite{macaulay2} 
with the package \href{http://www2.macaulay2.com/Macaulay2/doc/Macaulay2-1.13/share/doc/Macaulay2/Cremona/html/}{\emph{Cremona}} \cite{packageCremona}; 
some others, as
singular locus computations, are done using {\sc Singular} \cite{singular}.
We point out that the version 4.2.3 or later of the {\sc Macaulay2} package \emph{Cremona}
provides the method 
\href{http://www2.macaulay2.com/Macaulay2/doc/Macaulay2-1.13/share/doc/Macaulay2/Cremona/html/_special__Cubic__Transformation.html}{\texttt{specialCubicTransformation}}, 
which takes an integer between $1$ and $9$ and returns an explicit example of
special cubic birational transformation over $\mathbb{Q}$
in accordance to Table~\ref{table: SCTintCompl}.

\begin{example}\label{cubo-cubic into cubic}
 Let $Y=\PP^1\times \PP^1 \times \PP^1 \subset \PP^7$ be the segre embedding of three $\PP^1$. 
 It is classically known 
 (see \cite{edge,ciliberto-mella-russo}) that 
 $Y$ is a threefold with 
 \emph{one apparent double point}, that is for the general point $p\in \PP^7$ 
 there is a unique secant line $L\subset \PP^7$ to $Y$ passing through $p$. 
 Using this property, we can define a Cremona involution $T:\PP^7\dashrightarrow\PP^7$, by sending $p$
 to the point $T(p)\in L$ such that $\{p,T(p)\}$ 
 is harmonically conjugate to $L\cap Y$ (see e.g. \cite[Lecture~4]{dolgachev-cremonas}).
The dual variety to $Y$ is 
a  hypersurface  defined by a quartic form $F$, 
known as \emph{Cayley's hyperdeterminant}, 
and the involution $T$ coincides with the map defined by the partial derivatives of $F$.\footnote{For 
further computational details, see 
the online documentation of the methods 
\href{http://www2.macaulay2.com/Macaulay2/doc/Macaulay2-1.13/share/doc/Macaulay2/Cremona/html/_abstract__Rational__Map.html}{\texttt{abstractRationalMap}} 
from \href{http://www2.macaulay2.com/Macaulay2/doc/Macaulay2-1.13/share/doc/Macaulay2/Cremona/html/index.html}{\texttt{Cremona}}
\cite{packageCremona}, 
and 
\href{http://www2.macaulay2.com/Macaulay2/doc/Macaulay2-1.13/share/doc/Macaulay2/Resultants/html/_dual__Variety.html}{\texttt{dualVariety}}
from 
\href{http://www2.macaulay2.com/Macaulay2/doc/Macaulay2-1.13/share/doc/Macaulay2/Resultants/html/index.html}{\texttt{Resultants}}
\cite{packageResultants}.}
 The base locus of $T$ is a reducible fourfold $X\subset\PP^7$ 
 of degree $12$ and sectional arithmetic genus $10$ 
 consisting of 
  the union of three $\PP^1\times \PP^3\subset\PP^7$ 
intersecting pairwise in $Y$. Thus, 
although $T$ is far from being special, it has the same invariants as 
a hypothetical special Cremona transformation of $\PP^7$, by \cite[Theorem~0.6]{note4};
see also Remark~\ref{cubicP7}.

Now we deform $X$ by applying generic linkages.
Take $W$ to be a generic complete intersection  of three cubics containing $X$.
Then one has that $W = X\cup X'$, 
where $X'\subset\PP^7$ 
is an irreducible fourfold of degree $15$ and sectional genus $16$ 
cut out by $5$ cubics.
Take now $W'$ to be a generic complete intersection  of three cubics containing $X'$.
Then one has that $W' = X'\cup \widetilde{X}$, 
where $\widetilde{X}\subset\PP^7$ is
an irreducible 
fourfold of degree $12$ and sectional genus $10$, cut out 
by $9$ cubics that define a Cremona transformation $\widetilde{T}:\PP^7\dashrightarrow\PP^7$.
Unfortunately, $\widetilde{X}$ is still singular and hence $\widetilde{T}$ is not special.
However, the singular locus of $\widetilde{X}$ consists of only a finite number of points, 
so that the restriction of $\widetilde{T}$ to a general hyperplane $\PP^6\subset\PP^7$
gives us an example of special cubo-cubic birational transformation $\PP^6\dashrightarrow\ZZ\subset\PP^7$ 
into a cubic hypersurface $\ZZ$,
as in line VI of Table~\ref{table: SCTintCompl}. One verifies that the singular locus of $\ZZ$ has dimension one.
\end{example}

\begin{example}\label{cubo-quartic into quadric}
Let $\mathbb{S}^{10}\subset\PP^{15}$ be the
$10$-dimensional 
Spinor variety 
parametrizing one of the families of $\PP^4$'s contained in a smooth quadric in $\PP^9$.
Let $Y = \mathbb{S}^{10}\cap\PP^8\subset\PP^8$ be the intersection of $\mathbb{S}^{10}\subset\PP^{15}$
with a generic $\PP^8\subset\PP^{15}$.
So that $Y\subset\PP^8$ is a smooth threefold of degree $12$ and sectional genus $7$, cut out 
by $10$ quadrics (these $10$ quadrics define 
a special quadro-quartic birational transformation into a smooth quadric as in line XIII of Table~\ref{table: SCTintComplQuad}). 
Let $Y'\subset\PP^7$ be the projection of $Y$ in $\PP^7$ from a general point on $Y$.
Then $Y'$ is a smooth threefold of degree $11$ and sectional genus $7$, cut out by $5$ quadrics.
Take $W$ to be a generic complete intersection of type $(2,2,2,3)$ containing $Y'$.
We have that $W = Y'\cup X'$, where $X'\subset\PP^7$ is 
a smooth threefold of degree $13$ and sectional genus $10$, cut out by four quadrics and two cubics.
Let now $X\subset\PP^6$ be the projection of $X'$ in $\PP^6$ from a general point on $X'$.
Then $X\subset\PP^6$ is a smooth threefold of degree $12$ and sectional genus $10$, cut out by $8$ cubics.
These $8$ cubics define a special cubo-quartic birational transformation $\PP^6\dashrightarrow\ZZ\subset\PP^7$
into a quadric hypersurface $\ZZ$ singular along a line. Thus we have an example 
of transformation as in line VII of Table~\ref{table: SCTintCompl}.
\end{example}

\begin{example}\label{example G-1-5}
 Let $\mathbb{G}(1,5)\subset\PP^{14}$ be the Grassmanniann of lines in $\PP^5$.
 Let $Y = \mathbb{G}(1,5)\cap\PP^9\subset\PP^9$ be the intersection of $\mathbb{G}(1,5)\subset\PP^{14}$
with a generic $\PP^9\subset\PP^{14}$.
So that $Y\subset\PP^9$ is a smooth threefold of degree $14$ and sectional genus $8$, cut out 
by $15$ quadrics.
Let $X\subset\PP^6$ be the projection of $Y$ in $\PP^6$ from a plane  generated by 
three general points on $Y$.
Then $X\subset\PP^6$ is a smooth threefold of degree $11$ and sectional genus $8$, cut out by $9$ cubics.
These $9$ cubics define a special cubo-cubic birational transformation $\PP^6\dashrightarrow\ZZ\subset\PP^8$
into a complete intersection of two quadrics which has singular locus of dimension one.
Thus we have an example 
of transformation as in line VIII of Table~\ref{table: SCTintCompl}.
(Let us recall that the projection of $Y$ in $\PP^8$ from a general point on $Y$ 
is a threefold cut out by $9$ quadrics that define a special quadro-quintic Cremona transformation of $\PP^8$
as in line X of Table~\ref{table: SCTintComplQuad}.)
\end{example}

\begin{example}\label{cubo quadric into three quadrics}
 Here we give an example
of a special cubo-quadric birational transformation 
 into a complete intersection of three quadrics in $\PP^9$, as in line IX of Table~\ref{table: SCTintCompl}.
 This example has already been constructed in \cite[Subsection~4.4.2]{note4} but here 
 we provide 
 a slightly different construction.
 The base locus of this transformation 
 is a threefold scroll  in lines of degree $10$ and sectional genus $6$ whose
  existence  had been left undecided
 in \cite{fania-livorni-ten}.

 The main idea in \cite{note3}
 provides an algorithm to construct special birational transformations 
 whose base locus is a threefold scroll
 $X=\PP_Y(\mathcal{E})$ over a surface $Y$. In our specific case,
 the steps to follow are the following:
 \begin{itemize}
  \item Take the
  map $f:\PP^2\dashrightarrow\PP^{15}$ defined by the 
  linear system 
  of sextic curves with $4$ general double points. Then we have $Y=\overline{f(\PP^2)}$.
  \item Take $10$ general points
  $p_1,\ldots,p_{10}$
  on $\PP^2$ and let $h:\PP^{15}\dashrightarrow\PP^5$ be the projection 
  from the linear span of the $10$ points $f(p_1),\ldots,f(p_{10})$.
  Then $S = \overline{(h\circ f)(\PP^2)}$ is a general hyperplane section of $X$.
  \item  Let 
  $\psi:\PP^5\dashrightarrow\PP^{10}$ be 
  the map defined by the cubics through $S$. This is  a cubo-quadric birational map 
  into a complete intersection of $4$ quadrics. Then the linear system generated
  by these $4$ quadrics together with the $6$ quadrics defining the inverse of $\psi$ defines 
  a quadro-cubic Cremona transformation of $\PP^9$.
  \item Compute the inverse $\Psi:\PP^9\dashrightarrow\PP^9$
  of the Cremona transformation in the previous step. Then $\Psi$ is an extension map of $\psi$,
  and its base locus is an irreducible $6$-fold of degree $10$,
  sectional genus $6$, and with singular locus of dimension $2$.
  By restricting $\Psi$ to a general $\PP^6\subset\PP^9$ 
  we get a special birational transformation $\PP^6\dashrightarrow\ZZ\subset\PP^9$
  into a complete intersection  $\ZZ$ of three quadrics; the singular locus
  of $\ZZ$ has dimension one.
 \end{itemize}
\end{example}


\begin{example}[Subsection~4.2 of \cite{ein-shepherdbarron}; see also Subsection~4.1 of \cite{note4}]\label{example Cremona 14-15}
Let $\mathbb{G}(1,6)\subset\PP^{20}$ be the Grassmannian of lines in $\PP^6$,
and let $\mathrm{Sec}(\mathbb{G}(1,6))$ denote the variety of secant lines to $\mathbb{G}(1,6)$.
We have that $\mathrm{Sec}(\mathbb{G}(1,6))\subset\PP^{20}$ is an irreducible 
variety of codimension three, degree $14$, sectional genus $15$,
cut out by $8$ cubics, and with singular locus equal to $\mathbb{G}(1,6)$.
The intersection $\mathrm{Sec}(\mathbb{G}(1,6))\cap\PP^6\subset\PP^6$ 
of $\mathrm{Sec}(\mathbb{G}(1,6))$ with a generic $\PP^6\subset\PP^{20}$ is a smooth 
threefold of degree $14$ and sectional genus $15$ which is the base locus of a special 
cubo-quintic Cremona transformation of $\PP^6$, hence as in   line IV of Table~\ref{table: SCTintCompl}.
Notice that the equations of $\mathrm{Sec}(\mathbb{G}(1,6))\subset\PP^{20}$ 
are given by
the Pfaffians of the principal $6\times6$ minors of a generic $7\times7$ 
skew-symmetric matrix of variables.
\end{example}

\begin{example}[Subsection 4.2 of \cite{note4}]\label{example Cremona 13-12}
Here, 
we recall and simplify the construction
 of a special Cremona 
transformation of $\PP^6$ whose base locus is a threefold of degree $13$ and sectional genus $12$,
hence as in line V of Table~\ref{table: SCTintCompl}.

We first construct an irreducible singular threefold in $\PP^6$ of degree $13$ and sectional genus $12$.
 Let $C\subset\PP^3$ be the 
 isomorphic projection of the 
quintic rational normal curve 
together with an embedded point $p_0$, and 
let $p_1,\ldots,p_5$ be $5$ general points 
on a general plane passing through $p_0$.
Then the homogeneous ideal 
of the non-reduced scheme $C\cup\{p_1,\ldots,p_5\}\subset\PP^3$ is 
generated by $7$ quartics,
which define a map $\PP^3\dashrightarrow\PP^6$. The image of this map is
an irreducible  threefold $X\subset\PP^6$ of degree $13$ and sectional genus $12$, cut out by $7$ cubics,
and with singular locus of dimension two.

Now we proceed as in Example~\ref{cubo-cubic into cubic} in order to deform the threefold $X$ 
into another threefold $\widetilde{X}$.
Let $W\subset\PP^6$ be a generic complete intersection 
of three cubics containing $X$. Then we have $W = X\cup X'$, where
$X'\subset\PP^6$ is an irreducible threefold of degree $14$ and sectional genus $14$, cut out by $6$ cubics.
Now let $W'\subset\PP^6$ be a generic complete intersection of three cubics containing $X'$.
Then we have $W' = X'\cup\widetilde{X}$, where $\widetilde{X}\subset\PP^6$
is an irreducible threefold of degree $13$ and sectional genus $12$, cut out by $7$ cubics, 
and with only one singular point. 
Even if we repeat the procedure starting from $\widetilde{X}$ we still get one singular point. 

We remove the singularity of $\widetilde{X}$ in this way.
Take $S\subset\PP^5$ to be a generic hyperplane section of $\widetilde{X}$. 
The cubics through $S$ define a cubo-quintic birational map into a quintic hypersurface.
Taking together the quintic defining the image and the $6$ quintics defining the inverse, we get
a linear system of quintics on $\PP^6$ which defines a quinto-cubic Cremona transformation.
The inverse of this map has as base locus a smooth irreducible threefold of degree $13$
and sectional genus $12$.
\end{example}

\section{Connection with cubic fourfolds}\label{cubicFourfolds}

In all our examples of special cubic birational transformations of $\PP^6$ (see Table~\ref{table: SCTintCompl}),
a generic cubic  in the linear system defining the map 
turns out to be singular at a finite number of points.
Thus its generic hyperplane section $X\subset\PP^5$ 
is a smooth cubic hypersurface 
containing 
a very \emph{special} surface $S\subset\PP^5$, 
the general hyperplane section of the base locus of the map.

Let $\mathcal{S}$ be an irreducible component of the Hilbert scheme 
$\mathrm{Hilb}_{\mathbb{P}^5}^{\chi(\mathcal{O}_S(t))}$ 
containing $[S]$.
In lucky cases,
as it turns out to be in all our cases,
via some explicit calculations on a particular pair $(S,X)$,
and via a semicontinuity argument illustrated in \cite{nuer},
one can demonstrate that the
cubic hypersurfaces of $\PP^5$
containing a surface belonging to $\mathcal S$
form an irreducible locus in $\PP^{55}=\PP(H^0(\mathcal{O}_{\PP^5}(3)))$
 of codimension one.
 The closures in $\PP^{55}$ 
 of these loci belong to a countable family of divisors $\{C_{\delta}\}_{\delta}$
 indexed by the integers ${\delta}>6$ with ${\delta}\equiv0,2\ (\mathrm{mod}\, 6)$.
 These divisors have been introduced and studied by Hassett in \cite{Hassett,Has00} (see also \cite{Levico}).
 With our examples, we  cover the first four members of this family,
 that is $C_8$, $C_{12}$, $C_{14}$, and $C_{18}$,
 thus providing alternative geometric descriptions of them. 
 We include in Table~\ref{TabCongruenze} some more details.
 
 \begin{table}[htbp]
\centering   
\begin{tabular}{lllll}
\hline
 Transformation  & $\delta$ & $h^0(\mathcal{I}_{S/\PP^5}(3))$ & $h^0(\mathcal{N}_{S/\PP^5})$ & $h^0(\mathcal{N}_{S/X})$ \\
\hline 
Remark~\ref{rem Fano}  & $8$ & $10$ & $57$ & $12$ \\
Example~\ref{example Cremona 14-15} (Table~\ref{table: SCTintCompl}, line IV) & $14$ & $7$ & $77$ & $29$ \\
Example~\ref{example Cremona 13-12} (Table~\ref{table: SCTintCompl}, line V) & $14$ & $7$ & $68$ & $20$ \\
Example~\ref{cubo-cubic into cubic} (Table~\ref{table: SCTintCompl}, line VI) & $18$ & $8$ & $63$ & $16$ \\
Example~\ref{cubo-quartic into quadric} (Table~\ref{table: SCTintCompl}, line VII) & $12$ & $8$ & $65$ & $18$ \\
Example~\ref{example G-1-5} (Table~\ref{table: SCTintCompl}, line VIII) & $14$ & $9$ & $60$ & $14$ \\
Example~\ref{cubo quadric into three quadrics} (Table~\ref{table: SCTintCompl}, line IX) & $14$ & $10$ & $55$ & $10$ \\
\hline 
Example~\ref{example 7.3} & $18$ & $19$ & $39$ & $3$ \\
Example~\ref{example quadric in P13} & $18$ & $13$ & $51$ & $9$ \\
Example~\ref{threefold of degree 10 and sectional genus 7} & $14$ & $12$ & $58$ & $15$ \\
Example~\ref{example 7.6} & $14$ & $11$ & $63$ & $19$ \\
\hline 
\end{tabular}
 \caption{Surfaces $S\subset\PP^5$ contained in a cubic fourfold $[X]\in C_{\delta}$ 
 and obtained as general hyperplane section of the base locus
 of a 
 cubic 
 transformation of $\PP^6$.} 
\label{TabCongruenze} 
\end{table}

 The value of $\delta$  
 can be calculated in terms of the invariants
 of the surface $S$ (see \cite[Section~4.1]{Has00}).
Since we are assuming that $S$ is the general hyperplane section 
 of the base locus of a special cubic birational transformation,
 we can apply Lemma~\ref{numericalInvariants}\eqref{surface} to get a formula depending 
 only by the simplest invariants of the transformation.
 In the notation of Lemma~\ref{numericalInvariants}, 
 we have:
\begin{equation}\label{delta}
 \delta = -{\lambda}^2+6\,{\Delta}\,d-27\,{\lambda}+18\,g-36\,a  + 36\,\epsilon(d)  +288 .
\end{equation}
Notice that \eqref{delta} forces ${\delta}\equiv 0,2\ (\mathrm{mod}\, 6)$, but not $\delta>6$.
By imposing this last condition on the sets $\Gamma_{2480}^5$ and $\Gamma_{1139}^4$ given in 
Lemma~\ref{casifiniti}, 
and proceeding as in the proof 
of Theorem~\ref{thm main}, one can see
that no other values of $\delta$, besides $8,12,14,18$, can be achieved  
by a special birational 
transformation of $\PP^6$ 
of type $(3,d)$ into a factorial variety.

 One of the most challenging open problems in classical and modern algebraic geometry 
 is 
 the rationality of cubic fourfolds. 
 The works by Hassett, Kuznetsov, Addington, and Thomas (see \cite{kuz4fold,AT,kuz2,Levico}) 
 lead to the following
 problem (usually called \emph{Kuznetsov Conjecture}):
 \emph{The generic cubic fourfold $[X]\in C_{\delta}$ is rational 
 if and only if $\delta$ is an admissible value.}
 The admissible values are the even integers $\delta>6$ not 
 divisible by $4$, by $9$ and nor 
 by any odd prime of the form $2+3m$, so that the first admissible values are $14,26,38$.
 The rationality for $C_{14}$ 
 was shown in the classical works by Morin and Fano (see \cite{Morin,Fano}; see also \cite{BRS}), and in the recent 
 paper \cite{russo-stagliano-duke},
 Russo and ourselves showed
  the rationality for $C_{26}$ and $C_{38}$.
  The decisive step of our discovery was to find a description
 for $C_{\delta}$ 
 in terms of some surface $S$ contained in the generic $[X]\in C_{\delta}$ 
 and which admits (for some $e\geq1$) a \emph{congruence of $(3e-1)$-secant rational  curves of degree $e$},
 that is, 
 through a general point $p\in\PP^5$ there passes a unique rational  curve $C_p$ of degree $e$
 which is $(3e-1)$-secant to $S$; see \cite{russo-stagliano-duke} for precise and more general definitions.
 
 This property
 is a very rare phenomenon for the surfaces in $\PP^5$, 
 but not so rare for our surfaces.
 Indeed, the surfaces corresponding to the lines IV and V 
 of Table~\ref{table: SCTintCompl} 
 admit a congruence of $14$-secant  rational normal quintic curves (see \cite[Subsection~4.3]{Explicit});
 and the surface corresponding to the line VIII admits a congruence 
 of $8$-secant twisted cubics (see \cite[Subsection~4.2]{Explicit}).
 Furthermore, the isomorphic projection in $\PP^5$ of the surface corresponding to the line V of Table~\ref{table: SCTintComplQuad} 
 admits a congruence of $5$-secant conics (see \cite[Theorem~2]{russo-stagliano-duke});
 and the surface of $\PP^5$,
 base locus of the inverse map corresponding to the line II of Table~\ref{table: SCTintComplQuint}, 
 admits a congruence of $2$-secant lines \cite{Fano}.

\section{Special quadratic birational transformations}\label{sec quadratic}
For the convenience of the reader,
in this section we collect some of the results concerning 
special quadratic birational transformations.
For proofs and details, see \cite{note,note2,note3}.

\begin{theorem}[\cite{note,note2,note3}]\label{cor main quadratic}
Table~\ref{table: SCTintComplQuad} classifies all
the special birational transformations of type $(2,d)$, $d>1$, 
into a 
factorial complete intersection 
and whose base locus has dimension at most three. 
\end{theorem}
 \begin{table}[htbp]
\centering
\tabcolsep=7.3pt 
\begin{tabular}{llllllclll}  
\hline
& $r$ & $n$ & $a$ & $\lambda$ &  $g$  &  Abstract structure of $\B$ & $d$ & $\Delta$  & Existence   \\  
\hline
\rowcolor{Gray}
I &$1$ & $4$ & $0$ & $5$ & $1$ & Elliptic curve & $3$ & $1$ &   \\
II & $1$ & $4$ & $1$ & $4$ & $0$ & Rational normal curve & $2$ & $2$ &  \\ 
\rowcolor{Gray}
III & $2$ & $5$ & $0$ & $4$ & $0$  &  Veronese surface & $2$ & $1$ & \cite{semple-roth} \\ 
IV & $2$ & $6$ & $0$ & $7$ & $1$  & Elliptic scroll  & $4$ & $1$ & \cite{semple-tyrrell} \\ 
\rowcolor{Gray}
V & $2$ & $6$ & $0$ & $8$ & $3$  & Blow up of $\PP^2$ at $8$ simple points & $4$ & $1$ & \cite{cite2-semple} \\ 
VI & $2$ & $6$ & $1$ & $7$ & $2$  & \begin{tabular}{c}  Blow up of $\PP^2$ at $5$ simple points \\ and one double point \end{tabular}  & $3$ & $2$ & \cite{note} \\ 
\rowcolor{Gray}
VII & $2$ & $6$ & $2$ & $6$ & $1$  & Blow up of $\PP^2$ at $3$ simple points & $2$ & $4$ & \cite{note2} \\ 
VIII & $3$ & $7$ & $1$ & $6$ & $1$  & Hyperplane section of $\PP^2\times\PP^2\subset\PP^8$ & $2$ & $2$ & \cite{ein-shepherdbarron}  \\ 
\rowcolor{Gray}
IX & $3$ & $8$ & $0$ & $12$ & $6$  & Scroll over rational surf. with $K^2=5$   & $5$ & $1$  &  \cite{note3} \\
X & $3$ & $8$ & $0$ & $13$ & $8$    & \begin{tabular}{c} Blow up at one point \\ of $\mathbb{G}(1,5)\cap\PP^9\subset\PP^9$ \end{tabular}  & $5$ & $1$ & \cite{hulek-katz-schreyer}  \\
\rowcolor{Gray}
XI & $3$ & $8$ & $1$ &  $11$   &  $5$  &  Blow-up of $Q^3$ at $5$ points  & $3$ & $3$ & \cite{note} \\
XII & $3$ & $8$ & $1$ &  $11$   &  $5$  & Scroll over $\mathbb{F}_1$   & $4$ & $2$ & \cite{note3} \\
\rowcolor{Gray}
XIII & $3$ & $8$ & $1$ &  $12$   &  $7$ & Linear section of the spin. $\mathbb{S}^{10}\subset\PP^{15}$  & $4$ & $2$ &  \cite{ein-shepherdbarron} \\ 
XIV & $3$ & $8$ & $2$ & $10$ & $4$ & Scroll over $Q^2$ & $3$ & $4$  &  \cite{note3} \\
\rowcolor{Gray}
XV & $3$ & $8$ & $3$ & $9$ & $3$ & Scroll over $\PP^2$ & $2$ & $8$  &  \cite{note3} \\
\end{tabular}
 \caption{All the types of special birational transformations $\PP^n\dashrightarrow\ZZ\subseteq\PP^{n+a}$ 
 of type $(2,d)$, with $d>1$ and base locus $\B$ of dimension $r\leq 3$, degree $\lambda$, sectional genus $g$,
 into a factorial complete intersection $\ZZ\subseteq\PP^{n+a}$ of degree $\Delta$. } 
\label{table: SCTintComplQuad} 
\end{table}

\begin{remark}\label{rem quadratic}
Let $\varphi:\PP^n\dashrightarrow\ZZ\subseteq\PP^{n+a}$ be
  a special birational transformation of type $(2,d)$, with $d>1$ 
  and base locus of dimension three,
  into a prime Fano manifold $\ZZ$. 
  If $\ZZ$ is not a complete intersection, then
  we immediately deduce 
   from the classification in \cite[Table~1]{note3}
  that
  $\varphi$ is a quadro-quadric transformation of $\PP^8$ into $\mathbb{G}(1,5)\subset\PP^{14}$
  and its base locus is a rational normal scroll of degree $6$.
  This kind of transformations has been classically studied in \cite{semple}.
\end{remark}

By Remark~\ref{rem secants hypers}, 
the following proposition follows easily from results on \emph{quadratic entry locus varieties},
which are proved in \cite{russo-qel1,ciliberto-mella-russo,ionescu-russo-conicconnected}.
\begin{proposition}[\cite{note2}]\label{prop quadro-linear}
 Let $\varphi:\PP^n\dashrightarrow\ZZ\subseteq \PP^{n+a}$ be a special  birational transformation 
 of type $(2,1)$ into a  factorial
  complete intersection and whose base locus $\B$ has dimension $r\leq3$.
  Then $n\in\{3,4,5\}$, $\B\subset\PP^{n-1}\subset\PP^n$ is a quadric of codimension $2$,
  and $\ZZ\subset\PP^{n+1}$ is a quadric hypersurface.
\end{proposition}

We also point out that in the quadratic case, a number of results on special birational transformations 
have been obtained,
without imposing any condition on the dimension of the base locus.
Among these we have classifications for 
quadro-quadric Cremona transformations \cite{ein-shepherdbarron} (those whose base locus is a Severi variety \cite[Chapther~IV]{zak-tangent}),
quadro-cubic Cremona transformations \cite{russo-qel1},
quadro-quadric transformations into a smooth quadric \cite{note,alzati-sierra-quadratic},
quadro-quintic Cremona transformations \cite{note2,note3,russo-qel1},
quadro-cubic transformations into a factorial del Pezzo variety \cite{note2,note3},
quadro-linear transformations into a prime Fano manifold \cite{fu-hwang}.

\section{Special quartic and quintic birational transformations} \label{quintiche}
Taking into account
Corollary~\ref{casiProp},
Remark~\ref{analisi}, Corollary~\ref{cor main}, and Theorem~\ref{cor main quadratic},
 in order to complete the classification of the special birational 
 transformations $\varphi:\PP^n\dashrightarrow\ZZ\subseteq\PP^{n+a}$
 whose base locus $\B$ has dimension at most three and where $\ZZ$ is a
 factorial complete intersection, 
 we only need to classify the transformations as in the
 cases \eqref{nonquad2}, \eqref{caso3propA} and \eqref{caso3propANcre} of Corollary~\ref{casiProp}.
 We treat the case \eqref{nonquad2} in Proposition~\ref{propAS0}, the case \eqref{caso3propA} 
 in Proposition~\ref{propASquart}, and 
 the case \eqref{caso3propANcre} in Theorem~\ref{propAS}.
 We recall that these kinds of transformations  
 have been classified in \cite{alzati-sierra} under the hypothesis that $\ZZ$ is a smooth prime Fano variety.

\begin{proposition}\label{propAS0}
 Let $\varphi:\PP^4\dashrightarrow\ZZ\subseteq\PP^{4+a}$ be a special 
 birational transformation of type $(4,d)$, with $d>1$ and base locus $\B$ (necessarily of dimension two),
 into a non-degenerate factorial complete intersection $\ZZ\subseteq\PP^{4+a}$. Then one of the two following cases holds:
 \begin{itemize}
 \item $\varphi$ is a quarto-quartic Cremona transformation 
 and $\B$ is a determinantal surface of degree $\lambda=10$ and sectional genus $g=11$
 given by the vanishing of the $4\times 4$ minors of a $4\times 5$ matrix of linear forms;
 \item $\varphi$ is a quarto-quadric  transformation
 into a cubic hypersurface and $\B$ is a $K3$ surface
 with five $(-1)$-lines
 of degree 
   $\lambda=9$ and sectional genus $g=8$.
\end{itemize}
\end{proposition}
\begin{proof}
If $a=0$ we apply \cite[Theorem~3.3]{crauder-katz-1989}. So we can assume $a\geq 1$.
 Denoting by $\Delta$ the degree of $\ZZ$,
 by Corollary~\ref{casiProp}, case~\eqref{nonquad2},  we have to consider the following cases: 
 \begin{itemize}
 \item $d=3$, $a=1$, $\Delta=2$;
 \item $d=2$, $a=1$, $\Delta=3$;
 \item $d=2$, $a=2$, $\Delta=4$.
\end{itemize}
By \eqref{formuleHilbPol} we can express $\chi(\mathcal{O}_{\B})$ and the sectional genus $g$ 
as functions of $\lambda$ and $a$ as follows:
\begin{equation}\label{chiOS+g}
 \chi(\mathcal{O}_{\B}) = 6\,\lambda + 3\,a - 55 ,\quad g = 4\,\lambda + a - 29.
\end{equation}
Since $\B$ is a  surface embedded in $\PP^4$,
we can compute $K_{\B}^2$ and $c_2(\T_{\B})$ (see \cite[p.~434]{hartshorne-ag}):
\begin{equation}\label{KS2+c2T}
K_{\B}^2  = (1/2)\,{\lambda}^{2}+(27/2)\,{\lambda}+13\,a-180 ,\  
c_2(\T_{\B}) = -(1/2)\,{\lambda}^{2}+(117/2)\,{\lambda}+23\,a-480  .
\end{equation}
Using \eqref{KS2+c2T} one can easily compute that
\begin{equation}\label{segreNS}
 s_1(\N_{\B,\PP^4})\, H_{\B} =  -12\,{\lambda}-2\,a+60,\quad s_2(\N_{\B,\PP^4}) = -(1/2)\,{\lambda}^{2}+(217/2)\,{\lambda}+33\,a-780   .  
\end{equation}
From the formula \eqref{segresDegs} (with $(n,k)=(4,0)$) using \eqref{segreNS} we get
\begin{equation}\label{progDeg0}
 \Delta = (1/2)\,{\lambda}^{2}-(25/2)\,{\lambda} - a+76 .
\end{equation}
In the cases when $(a,\Delta)\in\{(1,2),(2,4)\}$, the equation \eqref{progDeg0} has no integral solutions,
while if $(a,\Delta) = (1,3)$ we get $\lambda=9$ or $\lambda=16$. 
Of course, $\lambda=16$ is impossible since $\B$ is cut out by quartics and is not a complete intersection,
therefore we have:
\begin{equation*}
 d = 2,\ a=1,\ \Delta=3,\ \lambda=9,\ g=8,\ K_{\B}^2=-5,\ \chi(\mathcal{O}_{\B}) = 2.
\end{equation*}
 Now the assertion follows from \cite{Aure}.
\end{proof}

 \begin{proposition}\label{propASquart}
 There are no 
 special 
 birational transformations $\PP^5\dashrightarrow\ZZ\subseteq\PP^{5+a}$ of type $(4,d)$, with $d>1$,
 into a
 linearly normal factorial variety $\ZZ\subseteq\PP^{5+a}$.
 
 Moreover, if  $\PP^5\dashrightarrow\ZZ\subseteq\PP^{5+a}$ is a special 
 birational transformation of type $(4,1)$ into a  factorial complete intersection $\ZZ$,
 then
 its base locus 
 is a threefold of 
  degree $9$ and sectional genus $9$, which is linked 
  to a cubic scroll in the complete intersection of a cubic and a quartic,
  and
 $\ZZ\subset\PP^8$ is a complete intersection of three quadrics.
\end{proposition}
 \begin{proof}
Let $\varphi:\PP^5\dashrightarrow\ZZ\subseteq\PP^{5+a}$ be a special birational transformation of type $(4,d)$
into a non-degenerate linearly normal factorial variety $\ZZ$ of degree $\Delta$.
From Proposition~\ref{PropDim}, we obtain that either  $d=2$ and $c(\ZZ) = 1$ (so that $\ZZ$ is a quadric hypersurface),
 or $d=1$ and $c(\ZZ) = 3$; moreover, in both cases the base locus $\B$ of $\varphi$
 is a threefold of a certain degree $\lambda$,
  and 
 the base locus of $\varphi^{-1}$ has dimension two.
 It follows that the multidegree of $\varphi$ is given by 
 \begin{equation}\label{multideg1}
 (1,\ 
 4,\ 
 16-\lambda,\ 
 \Delta\,d^2,\ 
 \Delta\,d,\ 
 \Delta) .
 \end{equation}
 Now, by Lemma~\ref{cond hilbert pol} we can express 
the Hilbert polynomial of $\B$ as a function of $\lambda$ and $a$.
In particular, the sectional genus of $\B$ is given by
$g = 4\,{\lambda}-2\,\epsilon(d) + a - 28$, where $\epsilon(1)=1$ and $\epsilon(2)=0$.
Then, just as in the proof of Lemma~\ref{numericalInvariants}, we can compute that
\begin{align*}
s_1(\N_{\B,\PP^5})\,H_{\B}^2 &= -12\,{\lambda}-2\,{a}+4\,\epsilon(d)+58, \\
s_2(\N_{\B,\PP^5})\,H_{\B} &= -{\Delta}\,{d}+96\,{\lambda}+32\,{a}-64\,\epsilon(d)-672, \\
s_3(\N_{\B,\PP^5}) &= 20\,{\Delta}\,{d}-640\,{\lambda}-{\Delta}-320\,{a}+640\,\epsilon(d)+5184,
\end{align*}
and therefore by \eqref{segresDegs} we deduce that the multidegree of $\varphi$ must be
\begin{equation}\label{multideg2} 
(1,\ 
4,\ 
16 -{\lambda},\ 
 2\,{a}-4\,\epsilon(d)+6,\ 
 {\Delta}\,{d},\ 
 {\Delta}) .
\end{equation}
By comparing \eqref{multideg1} and \eqref{multideg2}
we get:
\begin{equation}\label{quarto-codimIm}
  a = (1/2)\,{\Delta}\,{d}^2+2\,{\epsilon}(d)-3 .
\end{equation}
On the other hand, from the fact that the general hyperplane section of $\B$ is a surface embedded in $\PP^4$,
we get the following relation (see \cite[p.~434]{hartshorne-ag}):
\begin{equation}\label{quarto-Hart}
{\lambda}^2-2\,{\Delta}\,{d}-25\,{\lambda}-2\,{a}+16\,{\epsilon}(d)+150  = 0   .
\end{equation}
Thus, we conclude that $\Delta=d=2$ is impossible. Therefore it must be that
\begin{equation}\label{cond1quarto}
 d = 1,\quad {\Delta} = (1/3)\,{\lambda}^2-(25/3)\,{\lambda}+56, \quad a = (1/6)\,{\lambda}^2-(25/6)\,{\lambda}+27 ,
\end{equation}
and the first assertion follows.   

Now,  since we have $\lambda < 16$ and $g\geq 0$, we deduce that the invariants of $\varphi$ 
are as in one of the following cases ($m$ stands for multidegree):
\begin{enumerate}[(i)]
 \item\label{CasA1} $(\lambda,g,\Delta,a)= (6, 2, 18, 8)$, $m=(1, 4, 10, 18, 18, 18)$;
\item\label{CasA2} $(\lambda,g,\Delta,a)= (7, 4, 14, 6)$, $m=(1, 4, 9, 14, 14, 14)$;
\item\label{CasA3} $(\lambda,g,\Delta,a)= (9, 9, 8, 3)$, $m=(1, 4, 7, 8, 8, 8)$; 
\item\label{CasA4} $(\lambda,g,\Delta,a)= (10, 12, 6, 2)$, $m=(1, 4, 6, 6, 6, 6)$;
\item\label{CasA5} $(\lambda,g,\Delta,a)= (12, 19, 4, 1)$, $m=(1, 4, 4, 4, 4, 4)$;
\item\label{CasA6} $(\lambda,g,\Delta,a)= (13, 23, 4, 1)$, $m=(1, 4, 3, 4, 4, 4)$;
\item\label{CasA7} $(\lambda,g,\Delta,a)= (15, 32, 6, 2)$, $m=(1, 4, 1, 6, 6, 6)$.
\end{enumerate}
In case \eqref{CasA3} we can apply \cite{fania-livorni-nine}, and 
one easily verifies that this case occurs with $\ZZ$ smooth.
In cases \eqref{CasA1} and \eqref{CasA2},  $\ZZ$ is not a complete intersection.
Cases \eqref{CasA4}  and \eqref{CasA5} 
are in contradiction with the fact that the base locus of the inverse map has dimension (at least) two.
Finally, cases \eqref{CasA6} and \eqref{CasA7} are excluded since the multidegree does not satisfy the Hodge inequalities.
 \end{proof}

  The following proposition is established in a similar way as we showed Propositions~\ref{propAS0} and \ref{propASquart};
  see also \cite[Proposition~10]{alzati-sierra}.
 \begin{proposition}\label{prop quarto-linear}
 If $\PP^4\dashrightarrow\ZZ\subseteq\PP^{4+a}$ is a special 
 birational transformation of type $(4,1)$ into a  factorial complete intersection $\ZZ$,
 then
 its base locus 
 is a surface of 
  degree $9$ and sectional genus $9$, which is linked 
  to a cubic scroll in the complete intersection of a cubic and a quartic,
  and
 $\ZZ\subset\PP^7$ is a complete intersection of three quadrics.
 \end{proposition}

 We also omit the proof of the following result because 
 the methods we used in Section~\ref{cubiche} can be applied  here more efficiently, since the
 numerical invariants of 
 a threefold in $\PP^5$ satisfy more conditions of those
 of a threefold in $\PP^6$. Alternatively, one can obtain a proof 
 from the results of Alzati and Sierra in \cite{alzati-sierra}. 
 Indeed, even if they require that $\ZZ$ is a smooth prime Fano variety, 
 some of their arguments
 also apply to our situation with obvious changes.
\begin{theorem}[\cite{ein-shepherdbarron,alzati-sierra}]\label{propAS}
 Let $\varphi:\PP^5\dashrightarrow\ZZ\subseteq\PP^{5+a}$ be a special 
 birational transformation of type $(5,d)$ 
 into a non-degenerate factorial complete intersection $\ZZ\subseteq\PP^{5+a}$. Then its base 
   locus  $\B$ has dimension three, and
  one of the three following cases holds:
 \begin{itemize}
 \item $\varphi$ is a quinto-quintic Cremona transformation,
 $\B$ has degree $15$ and sectional genus $26$ and is 
 given by the vanishing of the $5\times 5$ minors of a $5\times 6$ matrix of linear forms;
\item  $\varphi$ is a quinto-cubic  transformation into a cubic hypersurface and 
  $\B$ has degree  $14$ and sectional genus $22$;
\item $\varphi$ is a quinto-linear transformation 
into a complete intersection of four quadrics and $\B$ has
 degree $14$ and sectional genus $23$.
  \end{itemize}
\end{theorem}
\begin{remark}\label{rem quintic abs struct}
In  the second case of Theorem~\ref{propAS}, $\B$ is linked inside 
the complete intersection of two quintics to a threefold of degree 
$11$ and sectional genus $13$, from which it follows that 
the ideal sheaf $\mathcal{I}_{\B}$ 
is given by a resolution
$
0\to \mathcal{T}_{\PP^5}(-7)\oplus \mathcal{O}_{\PP^5}(-6)\to\mathcal{O}_{\PP^5}(-5)^{\oplus 7}\to \mathcal{I}_{\B}\to 0 
$; see 
\cite{beltrametti_schneider_sommese_1992} and \cite[Lemma~7]{alzati-sierra}. 
Moreover the cubic hypersurface 
 $\ZZ\subset\PP^6$ turns out to be singular.
 In Example~\ref{exampleAS} we shall give an explicit example where $\ZZ$ is factorial, 
 being singular only at a finite number of points. 
 
 In the third case of Theorem~\ref{propAS},
 an explicit example can be constructed by using that 
$\B$ is linked inside the complete intersection of two quintics to a threefold 
of degree $11$ and sectional genus $14$, which is linked inside the 
complete intersection 
of two quartics to a threefold 
of degree $5$ and sectional genus $2$ as in Example~\ref{example 2.9}.
 Even in this case one sees that $\ZZ$ is  singular  only at a finite number of points.
 Moreover, 
the ideal sheaf $\mathcal{I}_{\B}$ 
is given by a resolution
$0\to \mathcal{O}_{\PP^5}(-6)^{\oplus 4} \to \mathcal{O}_{\PP^5}(-4)\oplus \mathcal{O}_{\PP^5}(-5)^{\oplus 4}\to \mathcal{I}_{\B}\to 0 
$; see 
\cite{beltrametti_schneider_sommese_1992} and \cite[Lemma~8]{alzati-sierra}.
\end{remark}

\begin{example}\label{exampleAS}
 Let $\varphi:\PP^6\dashrightarrow\PP^6$ be a special cubo-quintic Cremona transformation 
 as in Example~\ref{example Cremona 14-15}.
 Let $\ZZ\subset\PP^6$ 
 be a generic cubic hypersurface containing the base locus of $\varphi$.
 One sees that the singular locus of $\ZZ$ is zero-dimensional, and in particular $\ZZ$ is factorial.
 Now, the base locus of the inverse map 
 is 
 an irreducible fourfold of degree $14$, sectional genus $22$, and with zero-dimensional singular locus.
 Hence 
 the restriction of $\varphi$ to $\ZZ$ gives 
 a cubo-quintic birational transformation $\varphi|_{\ZZ}:\ZZ\dashrightarrow \PP^5\subset\PP^6$
 whose inverse map  is special and with three-dimensional base locus.
 Thus we have examples of transformations as in 
 line IV of Table~\ref{table: SCTintComplQuint}.
\end{example}

Some of the  relevant results of this section are summarized in Table~\ref{table: SCTintComplQuint}.
 \begin{table}[htbp]
\centering
\tabcolsep=7.6pt 
\begin{tabular}{llllllcllll}  
\hline
& $r$ & $n$ & $a$ & $\lambda$ &  $g$  &  Abstract structure of $\B$ & $d_1$ & $d$ & $\Delta$  & Existence   \\  
\hline
\rowcolor{Gray}
I &$2$ & $4$ & $0$ & $10$ & $11$ & Determinantal surface & $4$ & $4$ & $1$ &  \cite{crauder-katz-1989} \\
II &$2$ & $4$ & $1$ & $9$ & $8$ & K3 surface with $5$ $(-1)$-lines & $4$ & $2$ & $3$ &  \cite{Fano} \\
\rowcolor{Gray}
III &$3$ & $5$ & $0$ & $15$ & $26$ & Determinantal threefold & $5$ & $5$ & $1$ &  \cite{ein-shepherdbarron} \\
IV & $3$ & $5$ & $1$ & $14$ & $22$ & See Remark~\ref{rem quintic abs struct} & $5$ & $3$ & $3$ &  Example~\ref{exampleAS}
\end{tabular}
 \caption{All the types of special birational transformations $\PP^n\dashrightarrow\ZZ\subseteq\PP^{n+a}$ 
 of type $(d_1,d)$, with $d_1\in\{4,5\}$, $d>1$, and base locus $\B$ of dimension $r\leq 3$, degree $\lambda$, sectional genus $g$,
 into a factorial c. i. $\ZZ\subseteq\PP^{n+a}$ of degree $\Delta$. } 
\label{table: SCTintComplQuint} 
\end{table}
\begin{remark}\label{rem quintic}
 Let $\varphi:\PP^n\dashrightarrow\ZZ\subseteq\PP^{n+a}$ be
  a special birational transformation of type $(d_1,d)$, with $d_1\in\{4,5\}$, $d>1$, 
  and base locus of dimension three,
  into a prime Fano manifold $\ZZ$.
  If $\ZZ$ is not a complete intersection, then
   from  \cite[Theorem~8]{alzati-sierra} we deduce 
   that 
   $\varphi$ is a quinto-quadric transformation 
 into a linear section $\ZZ\subset\PP^{11}$
 of $\mathbb{G}(1,5)\subset\PP^{14}$ 
 and such that the 
 base locus $\B$ has degree $12$ and sectional genus $16$;
 moreover, the
 ideal sheaf $\mathcal{I}_{\B}$ is given 
 by a resolution 
 $
 0\rightarrow\O_{\PP^5}(-5)^{\oplus 3}\oplus\O_{\PP^5}(-6)\rightarrow \Omega_{\PP^5}(-3)\rightarrow \I_{\B}\rightarrow 0  
 $.
\end{remark}
 
\section{Special birational transformations whose inverse is linear}\label{sec linear}
We summarize some of our results in the following:
 \begin{theorem}
  Table~\ref{table: SCTlinear} classifies all
the special birational transformations of type $(d,1)$ 
into a 
factorial complete intersection 
and whose base locus has dimension at most three. 
 \end{theorem}
\begin{proof}
 The proof follows immediately by applying Proposition~\ref{PropDim} and then
  Theorem~\ref{thm cubo-linear},
  Proposition~\ref{prop quadro-linear},
  Proposition~\ref{propASquart},
  Proposition~\ref{prop quarto-linear},
 and Theorem~\ref{propAS}.
\end{proof}
 \begin{table}[htbp]
\centering
\tabcolsep=10.6pt 
\begin{tabular}{llllllllc}  
\hline
& $r$ & $n$ & $d$ & $\lambda$ &  $g$  & $a$ & $\Delta$ & Abstract structure of $\B$     \\  
\hline
\rowcolor{Gray}
I &$1$ & $3$ & $2$ & $2$ & $0$ & $1$ & $2$ & Conic \\
II &$1$ & $3$ & $3$ & $5$ & $2$ & $2$ & $4$ & Curve of genus $2$ \\
\rowcolor{Gray}
III &$2$ & $4$ & $2$ & $2$ & $0$ & $1$ & $2$ & Two-dimensional quadric \\
IV &$2$ & $4$ & $3$ & $5$ & $2$ & $2$ & $4$ & \begin{tabular}{c} Blow up of $\PP^2$ at $7$ simple \\ points  and one double point \end{tabular} \\
\rowcolor{Gray}
V & $2$ & $4$ & $4$ & $9$ & $9$ & $3$ & $8$ & \begin{tabular}{c} Surface linked to a cubic \\ scroll  in the c. i. of type $(3,4)$ \end{tabular} \\
VI &$3$ & $5$ & $2$ & $2$ & $0$ & $1$ & $2$ & Three-dimensional quadric \\
\rowcolor{Gray}
VII & $3$ & $5$ & $3$ & $5$ & $2$ & $2$ & $4$ & Quadric fibration over $\PP^1$ \\
VIII & $3$ & $5$ & $4$ & $9$ & $9$ & $3$ & $8$ & \begin{tabular}{c} Threefold linked to a cubic \\ scroll  in the c. i. of type $(3,4)$ \end{tabular} \\
\rowcolor{Gray}
IX & $3$ & $5$ & $5$ & $14$ & $23$ & $4$ & $16$ & See Remark~\ref{rem quintic abs struct}
\end{tabular}
 \caption{All the types of special birational transformations $\PP^n\dashrightarrow\ZZ\subset\PP^{n+a}$ 
 of type $(d,1)$ and base locus $\B$ of dimension $r\leq 3$, degree $\lambda$, sectional genus $g$,
 into a factorial complete intersection $\ZZ\subset\PP^{n+a}$ of degree $\Delta$. } 
\label{table: SCTlinear} 
\end{table}

 \begin{remark}
 Let $\varphi:\PP^n\dashrightarrow\ZZ\subset\PP^{n+a}$ be a special birational transformation of type $(d,1)$
 into a prime Fano manifold $\ZZ$ and whose base locus $\B$ has dimension three. 
 Assume further that $\ZZ$ is not a complete intersection. 
 Then we have $d\in\{2,3,4,5\}$,  by Proposition~\ref{PropDim}.
 
 If $d=2$, by \cite{fu-hwang} we deduce that 
 there are two types of transformations: either
 \begin{itemize}
 \item $\B=\PP^1\times \PP^2 \subset\PP^5\subset \PP^6$ and $\ZZ=\mathbb{G}(1,4)\subset\PP^9$; or
 \item $\B\subset\PP^6\subset\PP^7$ is a linear section of $\mathbb{G}(1,4)\subset\PP^9\subset\PP^{10}$ and $\ZZ\subset\PP^{12}$ is a linear section of the  Spinor variety $\mathbb{S}^{10}\subset\PP^{15}$.
 \end{itemize}
 
 If $d=3$, then one sees that $n=6$ and $\ZZ\subset \PP^{6+a}$ 
 has coindex $4$.
 The program used in Section~\ref{cubiche} 
 gives here a long list of not excluded cases. 
 Some of these cases occur really  but $\ZZ$ 
 is not necessarily smooth.
 Other cases can be easily excluded using the 
 classification of threefolds in $\PP^6$ of  degree at most twelve, see
 \cite{ionescu-smallinvariants,ionescu-smallinvariantsII,ionescu-smallinvariantsIII,
 fania-livorni-nine,fania-livorni-ten,besana-biancofiore-deg11,threefoldsdegree12}.
 We give a list of examples 
 in Subsection~\ref{subsection cubo-linear 2},
 which provides probably 
 the maximal list of types of 
 special cubo-linear birational transformations of $\PP^6$ into a 
 (factorial) variety $\ZZ\subset\PP^{6+a}$ and having base locus of dimension three.
 
 If $d=4$ or $d=5$, then we have $n=5$. 
 By \cite[Theorem~8]{alzati-sierra}
 we obtain that $d=5$, and that
 there are two types of transformations: either
 \begin{itemize}
  \item $\B$ has degree $13$ and sectional genus $19$, and $\ZZ\subset\PP^{10}$ 
  has 
 degree $21$ and coindex $4$; or 
  \item $\B$ has degree $11$ and sectional genus $13$, and
  $\ZZ\subset \PP^{15}$ is 
 a linear section of $\mathbb{G}(1,6)\subset\PP^{20}$.
 \end{itemize}
 An explicit example of such a threefold $\B\subset\PP^5$ of degree $13$ and sectional genus $19$
 can be constructed by taking the linked 
 inside the complete intersection 
 of a quartic and a quintic 
 to a septic Palatini scroll $\mathfrak P \subset \PP^5$ over
 a cubic surface in $\PP^3$ (see \cite{ionescu-smallinvariants,ottavianiScrolls}).
 By taking instead the linked $Y\subset\PP^5$ to $\mathfrak P \subset \PP^5$ 
 inside the complete intersection of two quartics, and then the linked 
 to $Y\subset\PP^5$ inside the complete intersection of a quartic and a quintic, 
 we get an explicit example of threefold of degree $11$ and sectional genus $13$ as above.
 \end{remark}

 \subsection{Further examples of special cubic transformations}\label{subsection cubo-linear 2}
 We conclude by giving four examples of special cubo-linear birational transformations 
 $\varphi:\PP^6\dashrightarrow\ZZ\subset\PP^{6+a}$
 of $\PP^6$ 
 into an irreducible variety $\ZZ\subset\PP^{6+a}$.
 In all the examples, the base locus $\B\subset\PP^6$ of $\varphi$ 
 is a smooth irreducible threefold and the base locus $\B'\subset\ZZ$ of $\varphi^{-1}$ is 
 an irreducible fourfold, equal to the image under $\varphi$ of the unique quadric hypersurface 
 containing $\B$.
 It turns out that the singular locus of $\ZZ$
 is strictly contained in $\B'$,
 as one can see by computing the tangent space 
 of $\ZZ$ at a general point of $\B'$.
 We are not able to check whether $\ZZ$ is factorial.
 
 \begin{example}\label{example 7.3}
 Let $Y\subset\PP^7$ be a smooth hyperplane section of $\PP^2\times \PP^2\subset\PP^8$.
 Let $X\subset\PP^6$ be the isomorphic projection of $Y$ 
 from a point not belonging to the secant variety of $Y$.
 Then $X$ is a smooth threefold 
 of degree $6$, sectional genus $1$, and cut out by $13$ cubics and one quadric.
 The linear system of cubics through $X$ 
 defines a birational map into a variety $\ZZ\subset\PP^{19}$ 
 of degree $57$ and cut out by $62$ quadrics. 
 Notice that from \cite{fujita-3-fold} it follows that there are no other special cubic birational 
 transformations of $\PP^6$ whose base locus is a non-linearly normal threefold.
 \end{example}
 
 \begin{example}\label{example quadric in P13}
  Let $Y\subset\PP^{13}\subset\PP^{14}$ be the image of a smooth quadric hypersurface in $\PP^4$
  under the quadratic Veronese embedding of $\PP^4$.
  Let $X\subset\PP^6$ be the projection of $Y$ in $\PP^6$ 
  from a linear subspace $\PP^6\subset\PP^{13}$ spanned by 
  $7$ general points on $Y$.
  Then one has that 
   $X$ is a smooth threefold of degree 
  $9$, sectional genus $5$, and cut out by $7$ cubics and one quadric.
   The linear system of cubics through $X$ 
 defines a birational map into a variety $\ZZ\subset\PP^{13}$ 
 of degree $30$ and cut out by $14$ quadrics.
 \end{example}
\begin{example}\label{threefold of degree 10 and sectional genus 7}
 Let $Y\subset\PP^8$ be a 
 general $3$-dimensional linear section of the $10$-dimensional 
Spinor variety  $\mathbb{S}^{10}\subset\PP^{15}$.
Let $X\subset\PP^6$ be the projection of $Y$ in $\PP^6$ from a line spanned by two general points on $Y$.
Then $X\subset\PP^6$ is a smooth threefold of degree $10$,
sectional genus $7$, and cut out by $6$ cubics and one quadric.
  The linear system of cubics through $X$ 
 defines a birational map into a variety $\ZZ\subset\PP^{12}$ 
 of degree $25$, cut out by $10$ quadrics, 
 and whose singular locus has dimension $1$.
\end{example}
\begin{example}\label{example 7.6}
Let $Y\subset\PP^6$ be a 
 general $3$-dimensional linear section of the Grassmannian $\mathbb{G}(1,4)\subset \PP^6$.
Take $W$ to be a generic complete intersection of two quadrics and one cubic containing $Y$,
and let $Y'=\overline{W\setminus Y}$ be the residual intersection. 
Take now 
$W'$ to be a generic complete intersection of one quadric and two cubics containing $Y'$, 
and let $X=\overline{W'\setminus Y'}$ be the residual intersection.
Then $X\subset\PP^6$ is a smooth threefold of degree $11$,
sectional genus $9$, and cut out by $5$ cubics and one quadric.
  The linear system of cubics through $X$ 
 defines a birational map into a variety $\ZZ\subset\PP^{11}$ 
 of degree $21$, cut out by $6$ quadrics and one cubic, 
 and whose singular locus has dimension $1$.
\end{example}
\begin{remark}
The existence of the threefold of degree $9$ and sectional genus $5$ given in 
Example~\ref{example quadric in P13}
had been left undecided
 in \cite{fania-livorni-nine}.

 The existence of the threefold of degree $10$ and sectional genus $7$ given in Example \ref{threefold of degree 10 and sectional genus 7} 
 reveals a missing case to the classification
  obtained in \cite{fania-livorni-ten} (it seems that the claim in \cite[Proposition~6.3]{fania-livorni-ten} is wrong).
\end{remark}

\subsection*{Acknowledgements}
The author is grateful to Francesco Russo for useful discussions and for his interest in the work.


\begin{thebibliography}{DGPS18}

\bibitem[AR]{Aure}
A.~B. Aure and K.~Ranestad, \emph{The smooth surfaces of degree $9$ in
  $\mathbb{P}^4$}, Complex Projective Geometry (G.~Ellingsrud, C.~Peskine,
  G.~Sacchiero, and S.~A. Stromme, eds.), Cambridge Univ. Press, pp.~32--46.

\bibitem[AS13]{alzati-sierra-quadratic}
A.~Alzati and J.~C. Sierra, \emph{Quadro-quadric special birational
  transformations of projective spaces}, Int. Math. Res. Not. IMRN \textbf{21}
  (2013).

\bibitem[AS16]{alzati-sierra}
\bysame, \emph{Special birational transformations of projective spaces}, Adv.
  Math. \textbf{289} (2016), 567--602.

\bibitem[AT14]{AT}
N.~Addington and R.~Thomas, \emph{Hodge theory and derived categories of cubic
  fourfolds}, Duke Math. J. \textbf{163} (2014), no.~10, 1886--1927.

\bibitem[BB05a]{besana-biancofiore-deg11}
G.~M. Besana and A.~Biancofiore, \emph{Degree eleven manifolds of dimension
  greater or equal to three}, Forum Math. \textbf{17} (2005), no.~5, 711--733.

\bibitem[BB05b]{besana-biancofiore-numerical}
\bysame, \emph{Numerical constraints for embedded projective manifolds}, Forum
  Math. \textbf{17} (2005), no.~4, 613--636.

\bibitem[BBS89]{beltrametti-biancofiore-sommese-loggeneral}
M.~Beltrametti, A.~Biancofiore, and A.~J. Sommese, \emph{Projective n-folds of
  log-general type. {I}}, Trans. Amer. Math. Soc. \textbf{314} (1989), no.~2,
  825--825.

\bibitem[BEL91]{bertram-ein-lazarsfeld}
A.~Bertram, L.~Ein, and R.~Lazarsfeld, \emph{Vanishing theorems, a theorem of
  {S}everi, and the equations defining projective varieties}, J. Amer. Math.
  Soc. \textbf{4} (1991), no.~3, 587--602.

\bibitem[BRS15]{BRS}
M.~Bolognesi, F.~Russo, and G.~Staglian\`o, \emph{Some loci of rational cubic
  fourfolds}, to appear in \emph{Math. Ann.}, preprint
  \url{http://arxiv.org/abs/1504.05863}, 2015.

\bibitem[BS95]{beltrametti-sommese}
M.~C. Beltrametti and A.~J. Sommese, \emph{The adjunction theory of complex
  projective varieties}, de Gruyter Exp. Math., vol.~16, Walter de Gruyter,
  Berlin, 1995.

\bibitem[BSS92]{beltrametti_schneider_sommese_1992}
M.~Beltrametti, M.~Schneider, and A.~J. Sommese, \emph{Threefolds of degree
  $11$ in $\mathbb{P}^5$}, Complex Projective Geometry: Selected Papers
  (G.~Ellingsrud, C.~Peskine, G.~Sacchiero, and S.~A. Stromme, eds.), London
  Math. Soc. Lecture Note Ser., Cambridge Univ. Press, 1992, p.~59–80.

\bibitem[BT15]{threefoldsdegree12}
M.~Bertolini and C.~Turrini, \emph{Threefolds in $\mathbb{P}^6$ of degree
  $12$}, Adv. Geom. \textbf{15} (2015), no.~2.

\bibitem[Cas89]{CastelnuovoBound}
G.~Castelnuovo, \emph{Ricerche di geometria sulle curve algebriche}, Atti R.
  Accad. Sci. Torino \textbf{24} (1889), 346--373.

\bibitem[CK89]{crauder-katz-1989}
B.~Crauder and S.~Katz, \emph{{C}remona transformations with smooth irreducible
  fundamental locus}, Amer. J. Math. \textbf{111} (1989), no.~2, 289--307.

\bibitem[CK91]{crauder-katz-1991}
\bysame, \emph{{C}remona transformations and {H}artshorne's conjecture}, Amer.
  J. Math. \textbf{113} (1991), no.~2, 269--285.

\bibitem[CMR04]{ciliberto-mella-russo}
C.~Ciliberto, M.~Mella, and F.~Russo, \emph{Varieties with one apparent double
  point}, J. Algebraic Geom. \textbf{13} (2004), no.~3, 475--512.

\bibitem[Deb01]{debarre}
O.~Debarre, \emph{Higher-dimensional algebraic geometry}, Universitext,
  Springer-Verlag, New York, 2001.

\bibitem[DGPS18]{singular}
W.~Decker, G.-M. Greuel, G.~Pfister, and H.~Sch{\"o}nemann, \emph{{\sc
  Singular} --- {A} computer algebra system for polynomial computations
  (version {4-1-1})}, Home page: \url{http://www.singular.uni-kl.de}, 2018.

\bibitem[Dol11]{dolgachev-cremonas}
I.~Dolgachev, \emph{Lectures on {C}remona transformations, {A}nn
  {A}rbor-{R}ome}, available at
  \url{http://www.math.lsa.umich.edu/~idolga/cremonalect.pdf}, 2010/2011.

\bibitem[Edg32]{edge}
W.~L. Edge, \emph{The number of apparent double points of certain loci}, Math.
  Proc. Cambridge Philos. Soc. \textbf{28} (1932), no.~3, 285--299.

\bibitem[ESB89]{ein-shepherdbarron}
L.~Ein and N.~Shepherd-Barron, \emph{Some special {C}remona transformations},
  Amer. J. Math. \textbf{111} (1989), no.~5, 783--800.

\bibitem[Fan43]{Fano}
G.~Fano, \emph{Sulle forme cubiche dello spazio a cinque dimensioni contenenti
  rigate razionali del $4^\circ$ ordine}, Comment. Math. Helv. \textbf{15}
  (1943), no.~1, 71--80.

\bibitem[FH18]{fu-hwang}
B.~Fu and J.~M. Hwang, \emph{Special birational transformations of type
  $(2,1)$}, J. Algebraic Geom. \textbf{27} (2018), 55--89.

\bibitem[FL94]{fania-livorni-nine}
M.~L. Fania and E.~L. Livorni, \emph{Degree nine manifolds of dimension greater
  than or equal to 3}, Math. Nachr. \textbf{169} (1994), no.~1, 117--134.

\bibitem[FL97]{fania-livorni-ten}
\bysame, \emph{Degree ten manifolds of dimension $n$ greater than or equal to
  3}, Math. Nachr. \textbf{188} (1997), no.~1, 79--108.

\bibitem[Fuj82]{fujita-3-fold}
T.~Fujita, \emph{Projective threefolds with small secant varieties}, Sci.
  Papers College Gen. Ed. Univ. Tokyo \textbf{32} (1982), 33--46.

\bibitem[Fuj90]{fujita-polarizedvarieties}
\bysame, \emph{Classification theories of polarized varieties}, London Math.
  Soc. Lecture Note Ser., vol. 155, Cambridge Univ. Press, Cambridge, 1990.

\bibitem[Ful84]{fulton-intersection}
W.~Fulton, \emph{Intersection theory}, Ergeb. Math. Grenzgeb. (3), no.~2,
  Springer-Verlag, 1984.

\bibitem[Gro68]{sga2}
A.~Grothendieck, \emph{Cohomologie locale des faisceaux coh{\'e}rents et
  th{\'e}or{\`e}mes de {L}efschetz locaux et globaux}, Adv. Stud. Pure Math.,
  vol.~2, North-Holland, Amsterdam, 1968, S{\'e}minaire de G{\'e}om{\'e}trie
  Alg{\'e}brique du Bois-Marie, 1962 (SGA 2).

\bibitem[GS19]{macaulay2}
D.~R. Grayson and M.~E. Stillman, \emph{{\sc Macaulay2} --- {A} software system
  for research in algebraic geometry (version 1.13)}, Home page:
  \url{http://www.math.uiuc.edu/Macaulay2/}, 2019.

\bibitem[Har70]{hartshorne-ample}
R.~Hartshorne, \emph{Ample subvarieties of algebraic varieties}, Lecture Notes
  in Math., vol. 156, Springer-Verlag, Berlin, 1970.

\bibitem[Har77]{hartshorne-ag}
\bysame, \emph{Algebraic geometry}, Grad. Texts in Math., vol.~52,
  Springer-Verlag, New York-Heidelberg, 1977.

\bibitem[Har94]{Hartshorne1994}
\bysame, \emph{Generalized divisors on {G}orenstein schemes}, {K}-{T}heory
  \textbf{8} (1994), no.~3, 287--339.

\bibitem[Has99]{Hassett}
B.~Hassett, \emph{Some rational cubic fourfolds}, J. Algebraic Geom. \textbf{8}
  (1999), no.~1, 103--114.

\bibitem[Has00]{Has00}
\bysame, \emph{Special cubic fourfolds}, Comp. Math. \textbf{120} (2000),
  no.~1, 1--23.

\bibitem[Has16]{Levico}
\bysame, \emph{Cubic fourfolds, {K3} surfaces, and rationality questions},
  Rationality Problems in Algebraic Geometry: Levico Terme, Italy 2015
  (R.~Pardini and G.~P. Pirola, eds.), Springer International Publishing, Cham,
  2016, pp.~29--66.

\bibitem[HKS92]{hulek-katz-schreyer}
K.~Hulek, S.~Katz, and F.-O. Schreyer, \emph{Cremona transformations and
  syzygies}, Math. Z. \textbf{209} (1992), no.~1, 419--443.

\bibitem[Ion84]{ionescu-smallinvariants}
P.~Ionescu, \emph{Embedded projective varieties of small invariants}, Algebraic
  Geometry Bucharest 1982, Lecture Notes in Math., vol. 1056, Springer-Verlag,
  Berlin, 1984, pp.~142--186.

\bibitem[Ion86a]{ionescu-smallinvariantsII}
\bysame, \emph{Embedded projective varieties of small invariants, {II}}, Rev.
  Roumaine Math. Pures Appl. \textbf{31} (1986), 539--544.

\bibitem[Ion86b]{ionescu-adjunction}
\bysame, \emph{Generalized adjunction and applications}, Math. Proc. Cambridge
  Philos. Soc. \textbf{99} (1986), no.~3, 457--472.

\bibitem[Ion90]{ionescu-smallinvariantsIII}
\bysame, \emph{Embedded projective varieties of small invariants, {III}},
  Algebraic Geometry, Lecture Notes in Math., vol. 1417, Springer-Verlag,
  Berlin, 1990, pp.~138--154.

\bibitem[IR10]{ionescu-russo-conicconnected}
P.~Ionescu and F.~Russo, \emph{Conic-connected manifolds}, J. Reine Angew.
  Math. \textbf{644} (2010), 145--157.

\bibitem[Kat87]{katz-cubo-cubic}
S.~Katz, \emph{The cubo-cubic transformation of $\mathbb{P}^3$ is very
  special}, Math. Z. \textbf{195} (1987), no.~2, 255--257.

\bibitem[KO73]{kobayashi-ochiai}
S.~Kobayashi and T.~Ochiai, \emph{Characterizations of complex projective
  spaces and hyperquadrics}, J. Math. Kyoto Univ. \textbf{13} (1973), no.~1,
  31--47.

\bibitem[Kuz10]{kuz4fold}
A.~Kuznetsov, \emph{Derived categories of cubic fourfolds}, Cohomological and
  Geometric Approaches to Rationality Problems, Progress in Mathematics, vol.
  282, Birkh\"{a}user Boston, 2010, pp.~219--243.

\bibitem[Kuz16]{kuz2}
\bysame, \emph{Derived categories view on rationality problems}, Rationality
  Problems in Algebraic Geometry: Levico Terme, Italy 2015 (R.~Pardini and
  G.~P. Pirola, eds.), Springer International Publishing, Cham, 2016,
  pp.~67--104.

\bibitem[Laz04]{lazarfeld-positivity}
R.~Lazarsfeld, \emph{Positivity in algebraic geometry. {I}}, Ergeb. Math.
  Grenzgeb. (3), vol.~48, Springer-Verlag, Berlin, 2004, Classical setting:
  line bundles and linear series.

\bibitem[LB80]{barz-quadrisecantes}
P.~{L}e {B}arz, \emph{Quadris\'{e}cantes d'une surface de $\mathbb{P}^5$}, C.
  R. Acad. Sci. Paris S\'{e}r. I Math. \textbf{291} (1980), 639--642.

\bibitem[LB81]{barz-multisecantes}
\bysame, \emph{Formules pour les multis\'{e}cantes des surfaces}, C. R. Acad.
  Sci. Paris S\'{e}r. I Math. \textbf{292} (1981), 797--800.

\bibitem[Li16]{li2016}
Q.~Li, \emph{Quadro-quadric special birational transformations from projective
  spaces to smooth complete intersections}, Internat. J. Math. \textbf{27}
  (2016), no.~1.

\bibitem[LS86]{livorni-sommese-15}
E.~L. Livorni and A.~J. Sommese, \emph{Threefolds of non negative {K}odaira
  dimension with sectional genus less than or equal to 15}, Ann. Sc. Norm.
  Super. Pisa Cl. Sci. \textbf{13} (1986), no.~4, 537--558.

\bibitem[Mor40]{Morin}
U.~Morin, \emph{Sulla razionalit\`a dell'ipersuperficie cubica dello spazio
  lineare {$S_5$}}, Rend. Semin. Mat. Univ. Padova \textbf{11} (1940),
  108--112.

\bibitem[Muk89]{mukai-biregularclassification}
S.~Mukai, \emph{Biregular classification of {F}ano 3-folds and {F}ano manifolds
  of coindex 3}, Proc. Natl. Acad. Sci. USA \textbf{86} (1989), no.~9,
  3000--3002.

\bibitem[Nue15]{nuer}
H.~Nuer, \emph{Unirationality of moduli spaces of special cubic fourfolds and
  {K3} surfaces}, Algebr. Geom. \textbf{4} (2015), 281--289.

\bibitem[Oko94]{okonek}
C.~Okonek, \emph{Notes on varieties of codimension 3 in $\mathbb{P}^n$},
  Manuscripta Math. \textbf{84} (1994), no.~1, 421--442.

\bibitem[Ott92]{ottavianiScrolls}
G.~Ottaviani, \emph{On $3$-folds in $\mathbb{P}^5$ which are scrolls}, Ann. Sc.
  Norm. Super. Pisa \textbf{19} (1992), 451--471.

\bibitem[RS17]{russo-stagliano-duke}
F.~Russo and G.~Staglian{\`o}, \emph{Congruences of $5$-secant conics and the
  rationality of some admissible cubic fourfolds}, to appear in \emph{Duke
  Math. J.}, preprint \url{https://arxiv.org/abs/1707.00999}, 2017.

\bibitem[RS18]{Explicit}
\bysame, \emph{Explicit rationality of some cubic fourfolds}, available at
  \url{https://arxiv.org/abs/1811.03502}, 2018.

\bibitem[Rus09]{russo-qel1}
F.~Russo, \emph{Varieties with quadratic entry locus, {I}}, Math. Ann.
  \textbf{344} (2009), no.~3, 597--617.

\bibitem[Sem31]{semple}
J.~G. Semple, \emph{On representations of the ${S}_k$'s of ${S}_n$ and of the
  {G}rassmann manifolds ${G}(k,n)$}, Proc. Lond. Math. Soc. \textbf{s2-32}
  (1931), no.~1, 200--221.

\bibitem[Sim04]{Simis2004162}
A.~Simis, \emph{Cremona transformations and some related algebras}, J. Algebra
  \textbf{280} (2004), no.~1, 162--179.

\bibitem[Som79]{sommesseDuke}
A.~J. Sommese, \emph{Hyperplane sections of projective surfaces {I} -- {T}he
  adjunction mapping}, Duke Math. J. \textbf{46} (1979), no.~2, 377--401.

\bibitem[Som86]{sommese-adjunction-theoretic}
\bysame, \emph{On the adjunction theoretic structure of projective varieties},
  Complex Analysis and Algebraic Geometry, Lecture Notes in Math., vol. 1194,
  Springer Berlin Heidelberg, 1986, pp.~175--213.

\bibitem[SR85]{semple-roth}
J.~G. Semple and L.~Roth, \emph{Introduction to algebraic geometry}, Oxford
  Univ. Press, New York, 1985, Reprint of the 1949 original.

\bibitem[ST68]{cite1-semple}
J.~G. Semple and J.~A. Tyrrell, \emph{Specialization of {C}remona
  transformations}, Mathematika \textbf{15} (1968), 171--177.

\bibitem[ST69]{semple-tyrrell}
\bysame, \emph{The {C}remona transformation of ${S}_6$ by quadrics through a
  normal elliptic septimic scroll $^1 {R}^7$}, Mathematika \textbf{16} (1969),
  no.~1, 88--97.

\bibitem[ST70]{cite2-semple}
\bysame, \emph{The ${T}_{2,4}$ of ${S}_6$ defined by a rational surface
  $^3{F}^8$}, Proc. Lond. Math. Soc. \textbf{s3-20} (1970), 205--221.

\bibitem[Sta12]{note}
G.~Staglian{\`o}, \emph{On special quadratic birational transformations of a
  projective space into a hypersurface}, Rend. Circ. Mat. Palermo \textbf{61}
  (2012), no.~3, 403--429.

\bibitem[Sta13]{note2}
\bysame, \emph{On special quadratic birational transformations whose base locus
  has dimension at most three}, Atti Accad. Naz. Lincei Cl. Sci. Fis. Mat.
  Natur. Rend. Lincei (9) Mat. Appl. \textbf{24} (2013), no.~3, 409--436.

\bibitem[Sta15]{note3}
\bysame, \emph{Examples of special quadratic birational transformations into
  complete intersections of quadrics}, J. Symbolic Comput. \textbf{74} (2015),
  635--649.

\bibitem[Sta16]{note4}
\bysame, \emph{Special cubic {C}remona transformations of $\mathbb{P}^6$ and
  $\mathbb{P}^7$}, Adv. Geom. (2016), { }In press, available at
  \url{http://arxiv.org/abs/1509.06028}.

\bibitem[Sta18a]{packageCremona}
\bysame, \emph{A {Macaulay2} package for computations with rational maps}, J.
  Softw. Alg. Geom. \textbf{8} (2018), no.~1, 61--70.

\bibitem[Sta18b]{packageResultants}
\bysame, \emph{A package for computations with classical resultants}, J. Softw.
  Alg. Geom. \textbf{8} (2018), no.~1, 21--30.

\bibitem[SV87]{sommese-adjunction-mapping}
A.~J. Sommese and A.~{Van de Ven}, \emph{On the adjunction mapping}, Math. Ann.
  \textbf{278} (1987), no.~1, 593--603.

\bibitem[Ver01]{vermeire}
P.~Vermeire, \emph{Some results on secant varieties leading to a geometric flip
  construction}, Compos. Math. \textbf{125} (2001), no.~3, 263--282.

\bibitem[Zak93]{zak-tangent}
F.~L. Zak, \emph{Tangents and secants of algebraic varieties}, Transl. Math.
  Monogr., vol. 127, Amer. Math. Soc., Providence, RI, 1993.

\end{thebibliography}

\providecommand{\bysame}{\leavevmode\hbox to3em{\hrulefill}\thinspace}
\providecommand{\MR}{\relax\ifhmode\unskip\space\fi MR }
\providecommand{\MRhref}[2]{%
  \href{http://www.ams.org/mathscinet-getitem?mr=#1}{#2}
}
\providecommand{\href}[2]{#2}

\end{document}